\documentclass[11pt]{amsart}
\usepackage{amsmath, amssymb, amsthm}
\usepackage{hyperref}
\usepackage{tikz}
\usepackage{subfigure}
\usepackage{ytableau}
\usepackage{array}
\usepackage{amsfonts}
\usepackage{mathtools}
\theoremstyle{plain}
\newtheorem{theorem}{Theorem}[section]
\newtheorem{lemma}[theorem]{Lemma}
\newtheorem{corollary}[theorem]{Corollary}
\newtheorem{proposition}[theorem]{Proposition}

\newtheorem{remark}[]{Remark}
\newtheorem{claim}[]{Claim}
\newtheorem{case}[]{Case}

\renewcommand{\Im}{\operatorname{Im}}

\numberwithin{equation}{section}

\begin{document}
	\title[Asymptotic Formula for $(t+1)$-Regular Partitions]{Asymptotic Formula for $(t+1)$-Regular Partitions} 
	\author{JAYANTA BARMAN}
	\address{JAYANTA BARMAN\\ Department of Mathematics \\
		Indian Institute of Technology Kharagpur \\
		Kharagpur-721302, India.} 
	\email{b1999jayanta@gmail.com}
	
	\author[Kamalakshya Mahatab]{Kamalakshya Mahatab}
	\address{Kamalakshya Mahatab\\ Department of Mathematics \\
		Indian Institute of Technology Kharagpur \\
		Kharagpur-721302, India.} 
	\email{kamalakshya@maths.iitkgp.ac.in}
	
	
	\subjclass[2020]{11P82, 11P55, 11N37}
	\keywords{ Partition Function, t-regular Partitions, Saddle Point Method}
	\begin{abstract}
		A partition is $t$-regular if none of its parts is divisible by $t$. Let $p(N,t)$ be the number of $(t+1)$-regular partitions of a positive integer $N$. In 1971, Hagis proved an asymptotic formula for $p(N,t)$ using the circle method, when $t$ fixed. In this article, we use the saddle point method and extend the result of Hagis in different ranges of $t$, obtaining explicit bounds. We also discuss an application of our result to estimate zeros in the character table of the symmetric group.
	\end{abstract}
	
	\maketitle
	\section{Introduction}\label{section1} 	
	A partition of a positive integer $N$ is a tuple $\lambda = (\lambda_1, \ldots, \lambda_\ell)$ of positive integers satisfying $\lambda_1 \ge \lambda_2 \ge \cdots \ge \lambda_\ell$ and $\sum_{i=1}^{\ell} \lambda_i = N$. Let $p(N)$ be the number of such partitions. Although the definition of a partition is straightforward, estimating its growth requires complex techniques. A breakthrough in this study was achieved in 1918 by Hardy and Ramanujan \cite{hardy1918asymptotic}, who established an asymptotic formula for $p(N)$ using the circle method. This result was later refined in 1938 by Rademacher \cite{rademacher1938partition}, leading to the following asymptotic expression:
	\begin{equation}\label{eq-hardy}
		p(N) = \frac{1}{4N\sqrt{3}} \exp\left( \frac{2\pi}{\sqrt{6}}\sqrt{N}\right) \left(1 + O\left(N^{-\frac{1}{2}}\right)\right).
	\end{equation}
	Similar to unrestricted partitions, many important arithmetic and combinatorial questions also arise from restricted partition functions. One classical example is that of 
	$t$-regular partitions:
	a partition $\lambda$ of $N$ is called $t$-regular if none of its parts is divisible by $t$. Let $p(N,t-1)$ denote the number of $t$-regular partitions of $N$. A classical result of Glaisher \cite{Glaisher1883} states that the number of $(t+1)$-regular partitions of $N$ is equal to the number of partitions of $N$ in which no part appears more than $t$ times. These partitions have been studied extensively and are related to many problems in combinatorics and number theory \cite{bhowmik2025number, mc2012hardy,mcspirit2023zeros}. Several properties of $p(N,t)$, including arithmetic properties,  multiplicative properties, and monotonicity, have been studied in \cite{barman2024arithmetic,beckwith2016multiplicative, keith2014congruences, singh2023proofs, singh2025hook}.

	Although restricted partition functions such as $p(N,t)$ were traditionally studied using combinatorial methods, Hagis \cite{hagis1963partitions, hagis1971partitions} was the first to derive a Rademacher-type \cite{rademacher1938partition} explicit formula for $p(N,t)$ with fixed $t \le N$, using the circle method. 
	In this article, we use the saddle point method to estimate $p(N,t)$, which refines Hagis' bound. Below, we recall the generating series and the integral formula for $p(N,t)$, which will be needed to apply the saddle point method.
	
	Let $q=\exp(2\pi iz)$, where $z=x+iy$, $y>0$. Recall that the Dedekind eta function $\eta(z)$ is defined by
	\begin{equation}\label{eq-eta}
		\eta(z)=\exp\left(\frac{\pi iz}{12}\right)\prod_{n=1}^{\infty}(1-\exp(2\pi inz)).   
	\end{equation}
	Following \cite{tyler2024asymptotics}, we introduce the functions
	\begin{equation}\label{demu(z)}
		\mu_{k}(z)=-\frac{z^{k+1}}{2\pi i} \left(\frac{d}{dz}\right)^{k} \log\eta(z),\quad \text{for $k\ge 1$}.
	\end{equation}
	The generating function for $p(N)$ \cite{euler1797introductio} is given by
	\begin{equation*}
		F(q)=\prod_{n=1}^{\infty}(1-q^{n})^{-1}=\sum_{N=0}^{\infty}p(N)q^{N}=q^{\frac{1}{24}}\frac{1}{\eta(z)}.
	\end{equation*}
	Similarly, the generating function for $p(N,t)$ (see \cite[(2.2)]{hagis1971partitions}) is 
	\begin{equation*}
		G(q,t)=\prod_{n=1}^{\infty}(1+q^{n}+q^{2n}+\cdots+q^{tn})=\frac{F(q)}{F(q^{t+1})}=\sum_{N=0}^{\infty}p(N,t)q^{N}.
	\end{equation*}
	
	By the Cauchy integral formula, the coefficient of $q^{N}$ in the above power series can be expressed as
	\begin{equation*}
		p(N,t)=\frac{1}{2\pi i}\int_{C}\frac{G(q,t)}{q^{N+1}}dq,
	\end{equation*}
	where $C$ is a simple, positively oriented closed contour around the origin and lies entirely inside the unit circle.  Let $C=\{z:q=\exp(2\pi i z)\}$. The requirement $|q| < 1$ is equivalent to the condition $y > 0$. Specifically, for a fixed $y$, as $x$ ranges over any interval of length $1$, the variable $q$ traces a complete circular contour of radius $e^{-2\pi y}$. Hence, it follows that
	\begin{equation}\label{eq-(1.1)}
		p(N,t)=\int_{-1/2}^{1/2}\exp\left(-2\pi izL\right)g_{t}(z)dx, 
	\end{equation}
	where
	\begin{align}\label{eq-gtz}
		g_{t}(z)=\frac{\eta((t+1)z)}{\eta(z)}\quad\text{and  } L=N+\frac{t}{24}.
	\end{align} 
	
	The above integral in (\ref{eq-(1.1)}) will be used in our saddle point analysis. We may note that the right-hand side of (\ref{eq-(1.1)}) depends on $y$, while the left-hand side is independent of $y$. This suggests that we optimize $y$, which may lead to an asymptotic formula for $p(N,t)$.
	\subsection{\texorpdfstring{Saddle point method}{}}
	Below is an outline of the saddle point method applied to a general sequence. 
	
	Let $b(N)$ be a sequence of positive numbers for all $N\ge 0$, with $b(0)=1$. Let the generating function associated with $b(N)$ be defined by
	\begin{align*}
		g(z)=\sum_{N=0}^{\infty}b(N)\exp(2\pi i Nz).
	\end{align*}
	So, we can write $b(N)$ as 
	\begin{align}\label{eq-intsad}
		b(N)=\int_{-1/2}^{1/2}\exp(-2\pi iNz)g(z)dx,
	\end{align}
	where $g(z)$ is assumed to have no zeros in the upper half plane.  Let
	\begin{align*}
		a(z)=-2\pi iNz+\log g(z).
	\end{align*}
	The saddle point $y$ is obtained by solving $a'(iy)=0$. Note that the Taylor expansion of $a(z)$ about $iy$ is given by
	\begin{align}\label{eq-Gtyl}
		\notag
		a(z)&=a(iy)+\frac{x^{2}}{2!}a''(iy)+\frac{x^{3}}{3!}a'''(iy)+  \frac{x^{4}}{4!}a''''(x'+iy)\\
		&=a(iy)+\frac{x^{2}}{2}a''(iy)\left(1+\frac{x}{3}\frac{a'''(iy)}{a''(iy)}+\frac{x^{2}}{12}\frac{a''''(x'+iy)}{a''(iy)}\right),
	\end{align}
	where $x'$ lies between $0$ and $x.$
	Consider a function of $y$, $e(y)< R$, where $R$ is the radius of convergence of the Taylor series in (\ref{eq-Gtyl}). With $e(y)$, we divide the range of integration in (\ref{eq-intsad}) into two parts:
	\begin{align*}
		|x|< e(y)\quad \text{and} \quad  e(y)\le |x|\le \frac{1}{2}.
	\end{align*}
	In the first region, we use the Taylor expansion of $a(z)$ to obtain the leading term. The contribution from the second region is absorbed into the error term. We may note
	\begin{align*}
		\int_{|x|< e(y)} \exp(a(z))dx=\exp(a(iy))\int_{|x|< e(y)}\exp\left(\frac{x^{2}}{2!}a''(iy)\left(1+\frac{x}{3}\frac{a'''(iy)}{a''(iy)}+\frac{x^{2}}{12}\frac{a''''(x'+iy)}{a''(iy)}\right)  \right) dx.
	\end{align*}
	Here, we bound $a'''(iy)$ and $a''''(iy)$ in terms of $a''(iy)$, which further contribute to the error term. Now, if we can show that $a''(iy) < 0$, then the remaining integral reduces to a Gaussian integral that can be evaluated easily. 
	
	
	We now state our main theorem, that gives an asymptotic formula for $p(N,t)$.
	
	\begin{theorem}\label{theorem-1.1}
		Let $N$ be a positive integer and $t \ge 4$. Then
		\begin{itemize}
			\item[\hypertarget{thm1.1i}{($i$)}] there exists a unique solution $y > 0$ satisfying
			\begin{equation}\label{eq-(1.3)}
				\frac{\mu_{1}((t+1)iy) - (t+1)\mu_{1}(iy)}{(t+1)y^{2}} = N + \frac{t}{24}. 
			\end{equation}
			Furthermore, $y$ lies in the range
			\begin{equation}\label{eq-(1.4)}
				\sqrt{\frac{t}{24(t+1)\left(N + \frac{t}{24}\right)}} < y < \frac{1}{\sqrt{24N - 1}}.
			\end{equation}
			\item[\hypertarget{thm1.1ii}{$(ii)$}] Corresponding to the above value of $y$,
				\begin{align*}
					p(N, t) =& \frac{\sqrt{t+1} y^{\frac{3}{2}}\exp\left(2\pi y\left(N + \frac{t}{24}\right)\right)\; \eta((t+1)iy)}{\sqrt{(t+1)\mu_{2}(iy) - \mu_{2}((t+1)iy)}\,\,\eta(iy)}\left(1+O(y)\right). 
				\end{align*} 
		\end{itemize}
	\end{theorem}
	However, the above formula is difficult to use in various applications directly, as $y$ is not explicit. In the next theorem, we solve this problem and obtain explicit expressions of $p(N,t)$ in different ranges of $t$.
	\begin{theorem}\label{theorem-1.2}
		Let N be a positive integer with $t+1\le N$. Then the following asymptotic formulas hold for $p(N,t)$:
		\begin{itemize}
			\item[\hypertarget{thm1.2i}{$(i)$}] If $(t + 1)\le \frac{2\pi}{(\frac{1}{2}+\epsilon)\log \left(N+\frac{t}{24}\right)}\sqrt{\frac{24(t+1)\left(N+\frac{t}{24}\right)}{t}} $ and $\epsilon>0$, then 
			\begin{equation*}
				p(N,t) =\frac{\sqrt{12}A_{t}(N)\exp\left(\frac{2\pi}{\sqrt{6}}\sqrt{\frac{t}{t+1}\left(N+\frac{t}{24}\right)}\right)}{\exp\left(\rho_1\exp\left(-2\pi\sqrt{\frac{(24N+t)}{t(t+1)}}+O\left(N^{-\frac{1}{2}}\right)\right)\right)}\left(1+O\left(N^{-\frac{1}{2}}\right)\right),
			\end{equation*}
			where $1 < \rho_{1} < 1.00873$ and 
			\begin{align*}
				A_{t}(N)=t^{\frac{1}{4}} \left((t+1)(24N+t)\right)^{-\frac{3}{4}}.
			\end{align*}
			\item[\hypertarget{thm1.2ii}{($ii$)}]
			If $\frac{2\pi}{(\frac{1}{2}+\epsilon)\log \left(N+\frac{t}{24}\right)}\sqrt{\frac{24(t+1)\left(N+\frac{t}{24}\right)}{t}}<t+1\le\sqrt{24N}$, where $A_t(N)$ is defined in $(i)$, then
			\begin{align*}
				p(N,t)&\le \frac{\sqrt{12}(2t+1)^{\frac{3}{4}}\exp\left(\frac{2\pi}{\sqrt{6}}\sqrt{\frac{2t+1}{2t+2}\left(N+\frac{t}{24}\right)}\right)}{\sqrt{t-1}\left((24N+t)(2t+2)\right)^{\frac{3}{4}}\exp\left(\exp\left(-2\pi\sqrt{\frac{24N+t}{t(t+1)}}\right)\right)}\left(1+O\left(N^{-\frac{1}{2}}\right)\right). 
			\end{align*}
			\item[\hypertarget{thm1.2iii}{($iii$)}] If $(t + 1) > \sqrt{24N}$, then 
			\begin{equation*}
				p(N,t)=p(N)\exp\left(-\nu_1\exp\left(-\frac{\pi (t+1)}{\sqrt{6N}}\left(1+O\left(N^{-\frac{1}{2}}\right)\right)\right)\right)\left(1+O\left(N^{-\frac{1}{2}}\right)\right),
			\end{equation*}
			where $1 < \nu_{1} < 1.00873$. 
		\end{itemize}
	\end{theorem}
	For $t > \sqrt{24N}$, the asymptotic behavior of $p(N,t)$ is relatively straightforward and closely approximates $p(N)$. However, the case where $t \le \sqrt{24N}$ exhibits more intricate behavior due to deeper links with the modular and analytic properties of $g_t(z)$.
	
	The following theorem is a special case of part~\hyperlink{thm1.2i}{$(i)$} of Theorem \ref{theorem-1.2}, and provides an asymptotic formula for $p(N,t)$ when $t$ is fixed, and coincides with Hagis' asymptotic formula~\cite{hagis1971partitions}.
	\begin{theorem}[{\cite[Corollary 4.2]{hagis1971partitions}}]\label{corollary-1.3}
		For fixed $t$, $4\le t \le N$, we have
		\begin{equation*}
			p(N,t)=\sqrt{12}t^{\frac{1}{4}}((t+1)(24N+t))^{-\frac{3}{4}}\exp\left(\frac{2\pi}{\sqrt{6}}\sqrt{\frac{t}{t+1}\left(N+\frac{t}{24}\right)}\right)\left(1+O\left(N^{-\frac{1}{2}}\right)\right),
		\end{equation*}
		as $N\to \infty.$
	\end{theorem}
	Using the above asymptotic result, McSpirit and Ono \cite{mcspirit2023zeros} established several results concerning zeros in the character tables of the symmetric groups $S_N$ for fixed $t$. Below, we discuss an application of Theorem~\ref{theorem-1.2} to estimate the number of zeros in the character table of the symmetric group. This improves our result in Theorem~1.3 of \cite{barman2025lower}.
	\subsection{\texorpdfstring{Application to Zeros in the character table of the symmetric group}{ Applications of Theorem 1.2}} 
	Given two partitions $\lambda$ and $\mu$ of a positive integer $N$. Let $\chi_{\lambda}(\mu)$ be the value of the irreducible character of the symmetric group $S_{N}$ labeled by $\lambda$, evaluated at the conjugacy class corresponding to $\mu$. As in \cite{mcspirit2023zeros}, we define
	\begin{equation*}
		Z_{t}(N) := \#\{(\lambda,\mu) : \chi_{\lambda}(\mu) = 0 \; \text{with $\lambda$ being a $t$-core}\}.
	\end{equation*}
	In the following theorem, we generalize the bounds of McSpirit and Ono~\cite{mcspirit2023zeros} for $Z_t(N)$ for arbitrary $t$. 
	The proof relies on Theorem \ref{theorem-1.2}, together with existing bounds for the number of $t$-core partitions $c_t(N)$, established in \cite{tyler2024asymptotics, tyler2024tcore} and subsequently simplified in \cite{barman2025lower}.
	\begin{theorem}\label{theorem-1.4}
		Let $N\ge 100$ be a large positive integer, and let $t \le N$. Let $A_{t}(N)$ be defined as in Theorem \ref{theorem-1.2} and $1<\rho_1, \nu_1<1.00873$. Then the following statements hold:
		\begin{itemize}
			\item[\hypertarget{thm1.4i}{($i$)}]For all $0<\epsilon<1$, if $6\leq t \leq \frac{2\pi\sqrt{2N}}{\sqrt{(1+\epsilon)\log N}}$, then
			\begin{equation*}
				Z_{t}(N)\ge c_{t}(N)(p(N)-p(N,t-1)), 
			\end{equation*}
			where 
			\begin{equation*}
				c_{t}(N)=\frac{(4\pi e)^{\frac{t-1}{2}}(t-1)}{\sqrt{4\pi}(t^{2}-t)^{\frac{t}{2}}} \left(N+\frac{t^{2}-1}
				{24}\right)^{\frac{t-3}{2}}(1+O(t^{-\epsilon}))   
			\end{equation*}
			and
			\begin{equation*}
				p(N,t-1)=\frac{\sqrt{12}A_{t-1}(N)\exp\left(\frac{2\pi}{\sqrt{6}}\sqrt{\frac{t-1}{t}\left(N+\frac{t-1}{24}\right)}\right)}{\exp\left(\rho_{1}\exp\left(-2\pi\sqrt{\frac{\left(24N+t-1\right)}{t(t-1)}}+O\left(N^{-\frac{1}{2}}\right)\right)\right)}\left(1+O\left(N^{-\frac{1}{2}}\right)\right).   
			\end{equation*} 
			\item[\hypertarget{thm1.4ii}{($ii$)}] For  
			$\frac{2\pi\sqrt{2N}}{\sqrt{\log N}}< t < f(N)$, where $f(N)\sim \frac{\sqrt{24N}}{\sqrt{\frac{6}{\pi}-1}}$, we have
			\begin{equation*}
				Z_{t}(N)\ge c_{t}(N)(p(N)-p(N,t-1)), 
			\end{equation*}
			where 
			\begin{equation*}
				c_{t}(N)\ge \frac{2\sqrt{\pi}\exp\left(\frac{t}{2}-1.00873te^{-2\pi}\right)\left(\frac{\pi}{6}(24N+t^{2}-1)\right)^{\frac{t-3}{2}}}{t^{t-1}} \left(1+O\left(t^{-1}\right)\right)   
			\end{equation*}
			and
			\begin{equation*}
				p(N,t-1)\le \frac{\sqrt{12}(2t-1)^{\frac{3}{4}}\exp\left(\frac{2\pi}{\sqrt{6}}\sqrt{\frac{2t-1}{2t}\left(N+\frac{t-1}{24}\right)}\right)}{\sqrt{t-2}\left((24N+t-1)2t\right)^{\frac{3}{4}}\exp\left(\exp\left(-2\pi\sqrt{\frac{24N+t-1}{t(t-1)}}\right)\right)}\left(1+O\left(N^{-\frac{1}{2}}\right)\right).   
			\end{equation*}
			\item[\hypertarget{thm1.4iii}{($iii$)}]  For $t \ge f(N)$, where $f(N)$ is defined in part (ii), we have
			\begin{align*}
				Z_{t}(N)\ge p(N)^{2}\exp\left(-1.00873t\exp\left(-\frac{t(t-1)}{2\left(N+\frac{t^{2}-1}{24}\right)}\right)\right)\left(1-\exp\left(-\nu_{1}\exp\left(-\frac{\pi t}{\sqrt{6N}}\right)\right)\right)\\
				\times\left(1+O\left(N^{-\frac{1}{2}}\right)\right).  
			\end{align*}
		\end{itemize}
	\end{theorem}
	The following corollary is a result of McSpirit and Ono \cite{mcspirit2023zeros} and follows immediately  from \hyperlink{thm1.4i}{${(i)}$} of Theorem~\ref{theorem-1.4}.
	
	\begin{corollary}[{\cite[Theorem 1.1]{mcspirit2023zeros}}]
		Let $t \ge 5$ be a fixed positive integer. Then, as $N \to \infty$, we have
		\begin{equation*}
			Z_{t}(N)\sim c_{t}(N)p(N)\gg_{t}N^{\frac{t-5}{2}} \exp\left(\frac{2\pi}{\sqrt{6}} \sqrt{N}\right).
		\end{equation*}
	\end{corollary}
	We can extend the range of $t$ in the above formula to
	$t \ll N^{\frac{1}{2}-\delta}$ for any $\delta>0$.
	
	\begin{corollary}
		Let $t\ll N^{\frac{1}{2}-\delta}$ be any positive integer, and $\delta>0$. Then
		\begin{equation*}
			Z_{t}(N)\sim c_{t}(N)p(N).   
		\end{equation*}
	\end{corollary}
	\begin{proof}
		From \hyperlink{thm1.2i}{$(i)$} of Theorem~\ref{theorem-1.2}, one can verify that  
		\begin{equation*}
			A_{t-1}(N)= O\left(\, t^{-\frac{1}{2}} N^{-\tfrac{3}{4}}\right).
		\end{equation*}
		From the bionomial expansion,
		\begin{align*}
			\sqrt{\frac{t-1}{t}\left(N+\frac{t-1}{24}\right)} &= \sqrt{N}-\left(\frac{\sqrt{N}}{2t}-\frac{(t-1)^{2}}{48t\sqrt{N}}\right)-\frac{\sqrt{N}}{8}\left(\frac{1}{t}-\frac{(t-1)^{2}}{24tN}\right)^{2}-\cdots .
		\end{align*}
		Taking $t \ll N^{\frac{1}{2}-\delta}$, we get
		\begin{align*}
			&\exp\left(\frac{2\pi}{\sqrt{6}}\sqrt{\frac{t-1}{t}\left(N+\frac{t-1}{24}\right)}\right)= \exp\left(\frac{2\pi}{\sqrt{6}}\sqrt{N}\right)\exp\left(-\frac{\pi\sqrt{N}}{t\sqrt{6}}\left(1+O\left(N^{-2\delta}\right)\right)\right)\\
			&\text{and}\quad \exp\left(\rho_1\exp\left(-2\pi\frac{\sqrt{24N+t-1}}{\sqrt{t(t-1)}}+O\left(N^{-\frac{1}{2}}\right)\right)\right)=1+O\left(\exp\left(-N^{\delta}\right)\right).
		\end{align*}
		Substituting the above estimates in \hyperlink{thm1.2i}{$(i)$} of Theorem~\ref{theorem-1.2}, we obtain 
		\begin{align*}
			p(N)-p(N,t-1)&=p(N)\left(1-\frac{p(N,t-1)}{p(N)}\right)\\
			&=p(N)\left(1+O\left(N^{\frac{1}{4}}t^{-\frac{1}{2}}\exp\left(-\frac{\pi\sqrt{N}}{t\sqrt{6}}\left(1+O\left(N^{-2\delta}\right)\right)\right)\right)\right).\\
		\end{align*}
		Therefore,
		\begin{equation*}
			c_{t}(N)\left(p(N)-p(N,t-1)\right)\le Z_{t}(N)\le c_{t}(N)p(N),
		\end{equation*}
		which implies 
		\begin{equation*}
			Z_{t}(N)\sim c_{t}(N)p(N).   
		\end{equation*}    
	\end{proof}
	\subsection{\texorpdfstring{Sketch of the proof of Theorem \ref{theorem-1.1}}{}}
	The proof starts by partitioning the range of the integration in \eqref{eq-(1.1)} into two distinct regions:
	\begin{align*}
		|x| < \frac{y}{3} \quad \text{and} \quad \frac{y}{3} \le |x| \le \frac{1}{2}.
	\end{align*}
	So, we can write equation (\ref{eq-(1.1)}) as
	\begin{equation}\label{eq-pNt}
		\begin{aligned}
			p(N,t)&=\int_{-y/3}^{y/3}\exp\left(-2\pi iz L\right)g_{t}(z)dx+\int_{\frac{y}{3}\le |x|\le \frac{1}{2}}\exp\left(-2\pi iz L\right)g_{t}(z)dx  \\
			&=\exp(2\pi Ly)g_{t}(iy)\int_{-y/3}^{y/3}\exp\left(-2\pi i Lx+2\pi i\frac{1}{2\pi i}\log\frac{g_{t}(z)}{g_{t}(iy)}\right)dx\\
			&+\exp(2\pi Ly)g_{t}(iy)\int_{\frac{y}{3}\le|x|\le \frac{1}{2}}\exp\left(-2\pi i Lx\right)\frac{g_{t}(z)}{g_{t}(iy)}dx.
		\end{aligned}
	\end{equation}
	The Taylor expansion of $\log g_t(z)$ around $x=0$ is given by
	\begin{align}\label{eq-tyler}
		\log \frac{g_{t}(z)}{g_{t}(iy)}&=\left(x\frac{d}{dz} \log g_{t}(z)+\frac{x^{2}}{2}\frac{d^{2}}{dz^{2}}\log g_{t}(z)+\cdots\right)_{z=iy}.
	\end{align}
	Using $g_t(z)$ from (\ref{eq-gtz}) and $\mu_k(z)$ from (\ref{demu(z)}), we obtain
	\begin{align*}
		\left(\frac{d}{dz}\right)^{k}\log g_{t}(z)=\frac{2\pi i}{(t+1)z^{k+1}}((t+1)\mu_{k}(z)-\mu_{k}((t+1)z)).
	\end{align*}
	From (\ref{eq-tyler}) and the above, we see that 
	\begin{align}\label{eq-taylor}
		\notag
		\frac{1}{2\pi i}\log \frac{g_{t}(z)}{g_{t}(iy)}&=x\frac{(t+1)\mu_{1}(iy)-\mu_{1}((t+1)iy)}{(t+1)(iy)^{2}}+\frac{x^{2}}{2!}\frac{(t+1)\mu_{2}(iy)-\mu_{2}((t+1)iy)}{(t+1)(iy)^{3}}+\cdots \\
		&=x\frac{(t+1)\mu_{1}(iy)-\mu_{1}((t+1)iy)}{(t+1)(iy)^{2}}+\frac{x^{2}}{2!}\frac{(t+1)\mu_{2}(iy)-\mu_{2}((t+1)iy)}{(t+1)(iy)^{3}}\\
		\notag
		&+\frac{x^{3}}{3!}\frac{(t+1)\mu_{3}(iy)-\mu_{3}((t+1)iy)}{(t+1)(iy)^{4}}+\frac{x^{4}}{4!}\frac{(t+1)\mu_{4}(z^{\prime})-\mu_{4}((t+1)z^{\prime})}{(t+1)(z^{\prime})^{5}},
	\end{align}
	for some $z^{\prime}=x^{\prime}+iy$ and $x^{\prime}$ is between $0$ and $x$.
	In Theorem \ref{theorem-main}, we demonstrate that the first integral of (\ref{eq-pNt}) provides the dominant contribution to the main term.  We solve the saddle point $y$ from the first order term of the Taylor series (\ref{eq-taylor}) by assuming $\frac{d}{dz}\log g_t(z)=2\pi iL$. The second order term gives the Gaussian integral contributing to the main term. The third and fourth-order terms are bounded by the second-order term and absorbed in the error.  
	
	In contrast, the second integral of (\ref{eq-pNt}) contributes only to the error term.	To control this error, we establish Proposition \ref{proposition-error}. This is the lengthiest part of our proof as the Taylor expansion of $g_t(z)$ is not valid in the region $y/3\leq x\leq 1/2$. In this region, we require `good' lower bound for $|\eta(z)|$ when $\Im(z)$ is small. This problem is solved by translating $z$ by $\text{SL}_{2}(\mathbb{Z})$ and then using the functional equation of $\eta(z)$.
	
	\section{Acknowledgments}
	J. Barman is deeply thankful to the University Grants Commission (UGC), India, for their invaluable support through the Fellowship Programme. K. Mahatab is supported by the DST INSPIRE Faculty Award Program (grant no. DST/INSPIRE/04/2019/002586) and ARG-MATRICS program (grant no. ANRF/ARGM/2025/002540/MTR).
	
	\section{Preliminaries}\label{sec-preli}
	In this section, we establish several auxiliary results that will be used in the proofs of our main results. The following proposition is used to tackle the integration of (\ref{eq-pNt}) in the region $\frac{y}{3}\le|x|\le\frac{1}{2}$. Since the Taylor expansion is not valid in this region, we instead derive suitable bounds for this contribution.
	\begin{proposition}\label{proposition-error}
		If $\frac{1}{y}\ge 1000$, then
		\begin{equation*}
			\int_{\frac{y}{3}\le |x|\le \frac{1}{2}}  \left|\frac{g_{t}(z)}{g_t(iy)}\right| dx<2\exp\left(-\frac{\pi}{150y}\right).
		\end{equation*}   
	\end{proposition}
	Note that the group $\text{SL}_{2}(\mathbb{Z})$ acts on the upper half–plane 
	$\mathbb{H}=\{\,z \mid \Im z>0\,\}$ through the Möbius 
	transformations, given by
	\begin{equation*}
		\gamma z=\frac{az+b}{cz+d} \quad \text{for}\quad \gamma= \begin{pmatrix}
			a & b\\
			c & d
		\end{pmatrix}\in \text{SL}_2(\mathbb{Z}).   
	\end{equation*}
	As shown in \cite{tyler2024asymptotics}, the Dedekind eta function (\ref{eq-eta}) obeys the following transformation identities
	\begin{equation}\label{eq-eta modular}
		\begin{aligned}
			\eta(z)&=\frac{1}{\sqrt{-iz}}\eta\left(\frac{-1}{z}\right)   \quad \text{and}\\
			|\eta(z)|&=\left(\frac{y}{\text{Im}\gamma z}\right)^{-\frac{1}{4}}|\eta(\gamma z)| \quad \text{for all }\quad \gamma\in \text{SL}_{2}(\mathbb{Z}).
		\end{aligned}
	\end{equation}
	Next, we state three lemmas that are needed to prove Proposition \ref{proposition-error}.
	The following lemma is used to bound $\eta(z)$.
	\begin{lemma}[{\cite[Lemma 2.2]{tyler2024asymptotics}}]\label{lemma-2.0}
		Suppose that $x\in\mathbb{R}$. For every $y \ge \frac{\sqrt{3}}{2}$, there exists a $\theta$ with $|\theta|<1$ such that
		\begin{align*}
			|\eta(z)|=\exp\left(-\frac{\pi y}{12}+\theta 1.01e^{-2\pi y}\right).
		\end{align*}
	\end{lemma}
	To estimate $\eta(iy)$ when $y$ is large, we apply the following lemma.
	\begin{lemma}[{\cite[Lemma 2.3]{barman2025lower}}]\label{lemma-2.1}
		For every $y\ge\frac{\sqrt{3}}{2}$, one can find a number $\nu$ satisfying $1<\nu<1.00873$ such that
		\begin{equation*}
			\eta(iy) = \exp\left(-\frac{\pi y}{12} - \nu e^{-2\pi y}\right).
		\end{equation*}
	\end{lemma}
	When $y$ is small and positive, we utilize the modular transformation property of $\eta(z)$, as described below.
	\begin{lemma}[{\cite[Lemma 2.4]{barman2025lower}}]\label{lemma-2.2}
		For any $0<y<1$, there exists a $\rho$ with $1<\rho<1.00873$ for which
		\begin{equation*}
			\eta(iy) = y^{-\frac{1}{2}} \exp\left( -\frac{\pi}{12y} - \rho e^{-\frac{2\pi}{y}} \right).
		\end{equation*}
	\end{lemma}
	We now give the proof of Proposition \ref{proposition-error}, using the preceding results. Throughout the proof, we assume that $\frac{1}{y}\ge 1000$.
	
	\begin{proof}[\textbf{\boldmath Proof of Proposition \ref{proposition-error}}]
		We obtain the following from the definition of $g_t(z)$ given in \eqref{eq-gtz}:
		\begin{align}\label{eq-gtz1}
			\left|\frac{g_{t}(z)}{g_t(iy)}\right|=&\left|\frac{\eta((t+1)z)}{\eta((t+1)iy)}\right| \left|\frac{\eta(iy)}{\eta(z)}\right|.
		\end{align}
		The main objective of this proof is to show that the above term is bounded by a sufficiently small quantity. 
		Since $\text{SL}_2(\mathbb{Z})$ acts transitively on the upper half-plane modulo the standard fundamental domain, we may select 
		$\gamma=\begin{pmatrix} a & b \\ c & d \end{pmatrix}\in \text{SL}_2(\mathbb{Z})$ 
		such that $\operatorname{Im}\gamma z\ge \frac{\sqrt{3}}{2}$.
		Changing $\gamma$ by $-\gamma$, which does not change its action on $z$, that is, $-\gamma z=\gamma z$. So, we can take $c\ge0.$ If $c=0$, then $\Im \gamma z=\Im z<\frac{\sqrt{3}}{2}$, which is a contradiction. Hence, we must have $c>0$. Without loss of generality, we can assume $0\le d<c$.
		Since in this proposition we assume that $y \le \frac{1}{1000}$, it follows from Lemma \ref{lemma-2.2} that
		\begin{align*}
			&\eta(iy)\le y^{-\frac{1}{2}}\exp\left(-\frac{\pi}{12y}+e^{-2000\pi}\right).  
		\end{align*}
		Since we assume that $\Im \gamma z \ge \frac{\sqrt{3}}{2}$, we can use equation \eqref{eq-eta modular} and Lemma \ref{lemma-2.0} to obtain
		\begin{align*}
			&|\eta(z)|=\left(\frac{y}{\text{Im}\gamma z}\right)^{-\frac{1}{4}}|\eta(\gamma z)|\ge\left(\frac{y}{\text{Im}\gamma z}\right)^{-\frac{1}{4}}\exp\left(-\frac{\pi}{12}\text{Im}\gamma z-1.01e^{-\sqrt{3}\pi}\right).   
		\end{align*}
		Combining the above arguments, we get
		\begin{equation}\label{eq-P2}
			\left|\frac{\eta(iy)}{\eta(z)}\right|<(y\text{Im}\gamma z)^{-\frac{1}{4}}\exp\left(-\frac{\pi}{12}\left(\frac{1}{y}-\text{Im}\gamma z\right)+\frac{1}{228}\right).  
		\end{equation}
		We now bound $\left|\frac{\eta((t+1)z)}{\eta((t+1)iy)}\right|$ by considering two cases, according to whether $(t+1)y \ge 1$ or $(t+1)y < 1$.
		\begin{case}\label{claim-3.11}
			\textbf{\boldmath $(t+1)y\ge 1$.} 
		\end{case}
		\begin{proof}[Proof of Case \ref{claim-3.11}]
			For $(t+1)y \ge 1>\frac{\sqrt{3}}{2}$, we use Lemma \ref{lemma-2.0} with $|\theta_j|<1$ for $j=1,2$:
			\begin{align}\label{eq-P3}
				\notag
				\left|\frac{\eta((t+1)z)}{\eta((t+1)iy)}\right|&=\exp(\theta_{1}1.01e^{-2\pi (t+1)y}-\theta_{2}1.01e^{-2\pi (t+1)y}) \\
				&\le\exp(2.02e^{-2\pi}).
			\end{align}
			Using (\ref{eq-P2}) and (\ref{eq-P3}) in (\ref{eq-gtz1}), we obtain
			\begin{align*}
				\left|\frac{g_{t}(z)}{g_t(iy)}\right|
				\le & (y\text{Im}\gamma z)^{-\frac{1}{4}}\exp\left(-\frac{\pi}{12y}+\frac{\pi}{12y}y\text{Im}\gamma z+\frac{1}{228}+2.02e^{-2\pi}\right). 
			\end{align*}
			Observe that $y\text{Im}\gamma z=\frac{y^{2}}{(cx+d)^{2}+c^{2}y^{2}}$. 
			When $c>1$, we have $y\text{Im}\gamma z=\frac{y^{2}}{(cx+d)^{2}+c^{2}y^{2}}<\frac{y^{2}}{c^{2}y^{2}}\le \frac{1}{4}$. 
			Otherwise, $c=1$ and $d=0$, in which case $y\text{Im}\gamma z=\frac{y^{2}}{x^{2}+y^{2}}\le \frac{9}{10}$, since $\frac{y}{3}\le |x|$. 
			Consider the function $v \longmapsto v^{-\frac{1}{4}}\exp\left(\frac{\pi v}{12y}\right)$, where $v=y\text{Im}\gamma z\in \left[\frac{\sqrt{3}y}{2}, \frac{9}{10}\right]$, because we have assumed that $\text{Im}\gamma z\ge \frac{\sqrt{3}}{2}$. The function is increasing in $v=y\Im \gamma z$, so the maximum is attained at the right endpoint. Hence,
			\begin{align*}
				\notag
				\left|\frac{g_{t}(z)}{g_t(iy)}\right| \le & \left(\frac{9}{10}\right)^{-\frac{1}{4}}\exp\left(-\frac{\pi}{12y}\left(1-\frac{9}{10}\right)+\frac{1}{228}+2.02e^{-2\pi}\right)\\
				< & 2 \exp\left(-\frac{\pi}{150y}\right).
			\end{align*}
			It follows from the above that
			\begin{equation*}
				\int_{\frac{y}{3}\le |x|\le \frac{1}{2}}  \left|\frac{g_{t}(z)}{g_t(iy)}\right| dx<2\exp\left(-\frac{\pi}{150y}\right).
			\end{equation*} 
		\end{proof}
		\begin{case}\label{claim-3.12}
			\textbf{\boldmath $(t+1)y < 1$.}
		\end{case}
		\begin{proof}[Proof of Case \ref{claim-3.12}]
			We apply Lemma \ref{lemma-2.2} and put an upper bound for $\rho$ to get
			\begin{align}\label{eq-C21}
				\eta((t+1)iy) \ge((t+1)y)^{-\frac{1}{2}}\exp\left(-\frac{\pi}{12(t+1)y}-1.01e^{-2\pi}\right).    
			\end{align}
			Furthermore, using the formula for $\frac{g_t(z)}{g_t(iy)}$ from \eqref{eq-gtz1}, the upper bound for $\frac{\eta(iy)}{\eta(z)}$ from \eqref{eq-P2}, and the lower bound for $\eta((t+1)iy)$ from \eqref{eq-C21}, we obtain
			\begin{align}\label{eq-C22}
				\left|\frac{g_{t}(z)}{g_t(iy)}\right|\le & |\eta((t+1)z)|(((t+1)y))^{\frac{1}{2}}(y\text{Im}\gamma z)^{-\frac{1}{4}}\exp\left(-\frac{\pi t}{12(t+1)y}+\frac{\pi}{12}\text{Im}\gamma z +\frac{1}{159}\right).
			\end{align}
			Since our goal is to show that $\frac{g_t(z)}{g_t(iy)}$ is small, it suffices to prove that either the term inside the exponent in (\ref{eq-C22}) is negative or that $\eta((t+1)z)$ is sufficiently small. To achieve this, we now divide the argument into two further cases, depending on whether $\text{Im}\gamma z< \frac{\sqrt{3}}{2}(t+1)$ or $\text{Im}\gamma z\ge \frac{\sqrt{3}}{2}(t+1)$. \newline
			\textbf{Subcase 2.1.} \textbf{\boldmath $\text{Im}\gamma z< \frac{\sqrt{3}}{2}(t+1)$.} 
			\newline
			We may apply Lemma 2.3 of \cite{tyler2024asymptotics} to bound $\eta((t+1)z)$ by $\frac{7}{9}((t+1)y)^{-\frac{1}{4}}$ and from (\ref{eq-C22}), we obtain
			\begin{align}\label{eq-C221}
				\left|\frac{g_{t}(z)}{g_t(iy)}\right|\le \frac{7}{9}(t+1)^{\frac{1}{4}}(\text{Im}\gamma z)^{-\frac{1}{4}}\exp\left(-\frac{\pi t}{12(t+1)y}+\frac{\pi}{12}\text{Im}\gamma z +\frac{1}{159}\right).
			\end{align}
			If we set $u=\text{Im}\gamma z$, then the quantity $u^{-\frac{1}{4}}\exp\left(\frac{\pi u}{12}\right)$ reaches its maximum over the interval $\frac{\sqrt{3}}{2} \le u < \frac{\sqrt{3}}{2}(t+1)$ at $u=\frac{\sqrt{3}}{2}(t+1)$, where its value is
			\begin{align*}
				u^{-\frac{1}{4}}\exp\left(\frac{\pi u}{12}\right)\le \left(\frac{2}{\sqrt{3}(t+1)}\right)^{\frac{1}{4}}\exp\left(\frac{\sqrt{3}\pi}{24}(t+1)\right).    
			\end{align*}
			Since $(t+1)y<1$, for $t\ge 29$, it follows from \eqref{eq-C221} and the preceding argument that
			\begin{align}\label{eq-C222}
				\left|\frac{g_{t}(z)}{g_t(iy)}\right|&\le \exp\left(-\frac{\pi t}{12(t+1)y}+\frac{\sqrt{3}\pi}{24y}\right)\le \exp\left(-\frac{\pi}{150y}\right).
			\end{align}
			For $t<29$, the same bound follows from $\exp\left(\frac{\sqrt{3}\pi}{24}(t+1)\right)\le \exp\left(\frac{29\sqrt{3}\pi}{24}\right)$ and $y\le \frac{1}{1000}$.
			\newline
			\textbf{Subcase 2.2.} \textbf{\boldmath $\operatorname{Im}\gamma z \ge \frac{\sqrt{3}}{2}(t+1)$.} 
			\newline
			We denote the greatest common divisor $\gcd(a,b)$ by $(a,b)$. Since $\gamma\in \text{SL}_2(\mathbb{Z})$ implies $(c,d)=1$, it follows that for any positive integer $t$, $\frac{c}{(t+1,c)}$ and $\frac{d(t+1)}{(t+1,c)}$ are coprime. 
			So, $\gamma_{1}=
			\begin{pmatrix}
				\bullet & \bullet\\
				c/(t+1,c) & d(t+1)/(t+1,c)
			\end{pmatrix}\in \text{SL}_2(\mathbb{Z})$ exist. Then from (\ref{eq-eta modular}), 
			\begin{equation}\label{eq-SC21}
				|\eta((t+1)z)|=\left(\frac{(t+1)y}{\text{Im}\gamma_1(t+1) z}\right)^{-\frac{1}{4}} |\eta(\gamma_{1}(t+1)z)| 
			\end{equation}
			and
			\begin{align*}
				\text{Im}\gamma_1(t+1)z=\frac{(t+1)y}{\left|\frac{c(t+1)z}{(t+1,c)}+\frac{d(t+1)}{(t+1,c)}\right|^{2}} =\frac{(t+1,c)^{2}}{t+1}\text{Im}\gamma z\ge \frac{\sqrt{3}}{2} .
			\end{align*}
			Therefore, we may apply the Lemma \ref{lemma-2.0} to $\eta(\gamma_{1}(t+1)z)$. 
			Substituting the above estimate in (\ref{eq-SC21}) and using Lemma \ref{lemma-2.0}, we deduce
			\begin{align*}
				|\eta((t+1)z)| \le \sqrt{\frac{(t+1,c)}{(t+1)}}\left(\frac{y}{\text{Im}\gamma z}\right)^{-\frac{1}{4}}\exp\left(-\frac{\pi(t+1,c)^{2}}{12(t+1)}\text{Im}\gamma z+1.01e^{-\sqrt{3}\pi}\right).
			\end{align*}
			So, using the above bound for $\eta((t+1)z)$ in (\ref{eq-C22}), we obtain the following
			\begin{align}\label{eq-SC24}
				\left|\frac{g_{t}(z)}{g_t(iy)}\right|\le ((t+1,c))^{\frac{1}{2}}\exp\left(-\frac{\pi}{12}\left(\frac{(t+1,c)^{2}}{(t+1)}-1\right)\text{Im}\gamma z+\frac{1}{90}\right)\exp\left(-\frac{\pi t}{12(t+1)y}\right).  
			\end{align}
			As mentioned earlier, our objective is to show that $\frac{g_t(z)}{g_t(iy)}$ is small. We therefore consider two possibilities according to whether $(t+1,c)^2 > (t+1)$ or $(t+1,c)^2 \le (t+1)$. 
			We refer to case $(t+1,c)^2 > (t+1)$ as Subcase~2.2.1, where the desired bound follows easily. We treat case $(t+1,c)^2 \le (t+1)$ as Subcase~2.2.2. This case is more delicate, as the dominant exponential term 
			in \eqref{eq-SC24} cannot be identified directly.
			To handle this, we introduce suitable notation and decompose the integral into smaller parts in the region $\Im \gamma z \ge \frac{\sqrt{3}}{2}(t+1)$.
			\newline
			\noindent\textbf{Subcase 2.2.1.} \textbf{\boldmath $(t+1,c)^{2} > (t+1)$.} 
			\newline
			The right hand side of (\ref{eq-SC24}) is maximized at $(t+1,c)^2=(t+2)$. Hence,
			\begin{align*}
				\left|\frac{g_{t}(z)}{g_t(iy)}\right|&\le (t+2)^{\frac{1}{2}}\exp\left(-\frac{\pi}{12}\left(\frac{t+2}{(t+1)}-1\right)\text{Im}\gamma z+\frac{1}{90}\right)\exp\left(-\frac{\pi t}{12(t+1)y}\right).
			\end{align*}
			Moreover, in the range $\text{Im}\gamma z\ge \frac{\sqrt{3}}{2}(t+1),$ the right-hand side reaches its maximum at  $\text{Im}\gamma z= \frac{\sqrt{3}}{2}(t+1)$. Therefore,
			\begin{align}\label{eq-SC211}
				\notag
				\left|\frac{g_{t}(z)}{g_t(iy)}\right|  &\le \left(t+2\right)^{\frac{1}{2}}\exp\left(-\frac{\sqrt{3}\pi }{24}+\frac{1}{90}\right)\exp\left(-\frac{\pi t}{12(t+1)y}\right)\\
				&\le \left(t+2\right)^{\frac{1}{2}}\exp\left(-\frac{\pi t}{12(t+1)y}\right).
			\end{align}
			\noindent\textbf{Subcase 2.2.2.} \textbf{\boldmath $(t+1,c)^{2} \le (t+1)$.} 
			\newline
			Note,
			\begin{align*}
				((t+1,c))^{\frac{1}{2}}\exp\left(\frac{\pi}{12}\left(1-\frac{(t+1,c)^{2}}{(t+1)}\right)\text{Im}\gamma z+\frac{1}{90}\right)\le 1.02\exp\left(\frac{\pi t}{12(t+1)}\text{Im}\gamma z\right),   
			\end{align*}
			as the maximum is attained at $(t+1,c)=1$. From (\ref{eq-SC24}), we simplify
			\begin{align}\label{eq-SSC21}
				\left|\frac{g_{t}(z)}{g_t(iy)}\right|\le 1.02 \exp\left(\frac{\pi t}{12(t+1)}\text{Im}\gamma z\right)\exp\left(-\frac{\pi t}{12(t+1)y}\right).  
			\end{align}
			Since we only know that $\Im \gamma z \ge \frac{\sqrt{3}}{2}(t+1)$, this condition alone does not determine which of the two terms, 
			$\exp\left(\frac{\pi t}{12(t+1)} \Im \gamma z\right)$ or 
			$\exp\left(-\frac{\pi t}{12(t+1)y}\right)$ in \eqref{eq-SSC21}, is dominant. To address this difficulty, we introduce the following notation, which helps us to break the integral into smaller parts where $\Im\gamma z \ge \frac{\sqrt{3}}{2}(t+1)$. Let 
			\begin{align}\label{eq-SSC22}
				m(c,d)=\left\{x\,|\,\text{Im}\gamma z\ge\frac{\sqrt{3}(t+1)}{2}\right\}=\left\{x\,|\left|\frac{x+\frac{d}{c}}{y}\right|<\sqrt{\frac{2}{\sqrt{3}c^{2}(t+1)y}-1}\right\}
			\end{align}
			denote the portions of the interval into which we break the integration for ease of calculation.
			\newline
			Now we estimate the integral $\int_{\frac{y}{3}\le|x|\le \frac{1}{2}}\left|\frac{g_t(z)}{g_t(iy)}\right|dx$ corresponding to Case \ref{claim-3.12}. If $\Im\gamma z< \frac{\sqrt{3}}{2}(t+1)$, we apply the bound from \eqref{eq-C222}. If $\Im\gamma z\ge \frac{\sqrt{3}}{2}(t+1)$ and $(t+1,c)^{2} \le (t+1)$, we use the bound of \eqref{eq-SSC21}. Among the bounds in \eqref{eq-SC211} and \eqref{eq-SSC21}, we take the larger one as our final estimate. If $\Im\gamma z\ge \frac{\sqrt{3}}{2}(t+1)$ and $(t+1,c)^{2} \le (t+1)$, we break the integral into smaller regions, $m(c,d)$, defined in (\ref{eq-SSC22}). Therefore,
			\begin{align}\label{eq-SSC23}
				\notag
				\int_{\frac{y}{3}\le|x|\le \frac{1}{2}}\left|\frac{g_{t}(z)}{g_t(iy)}\right|dx&\le \int_{\Im \gamma z\ge \frac{\sqrt{3}}{2}(t+1)}\left|\frac{g_{t}(z)}{g_t(iy)}\right|dx+\int_{\Im \gamma z< \frac{\sqrt{3}}{2}(t+1)}\left|\frac{g_{t}(z)}{g_t(iy)}\right|dx\\
				&\le \int_{\substack{m(1,0)\\|x|>\frac{y}{3}}}\left|\frac{g_{t}(z)}{g_t(iy)}\right|dx+\sum_{c=2}^{\infty}\sum_{\substack{d=0\\(c,d)=1}}^{c-1}\int_{m(c,d)}\left|\frac{g_{t}(z)}{g_t(iy)}\right|dx+\exp\left(-\frac{\pi}{150y}\right).
			\end{align}
			We now simplify $\Im\gamma z$ for $c = 1$ and $d = 0$. 
			From \eqref{eq-SSC21}, we obtain
			\begin{align}
				\notag
				\int_{\substack{m(1,0)\\|x|>\frac{y}{3}}}\left|\frac{g_{t}(z)}{g_t(iy)}\right|dx &\le 1.02  \exp\left(-\frac{\pi t}{12(t+1)y}\right) \int_{\substack{m(1,0)\\|x|>\frac{y}{3}}}\exp\left(\frac{\pi t}{12(t+1)}\text{Im}\gamma z\right)\\
				\notag
				&=1.02\exp\left(-\frac{\pi t}{12(t+1)y}\right) \int_{\substack{m(1,0)\\|x|>\frac{y}{3}}}\exp\left(\frac{\pi t}{12(t+1)y}\frac{y^{2}}{x^{2}+y^{2}}\right)dx.
			\end{align}
			Let $v = 1 + \left( \frac{x}{y} \right)^{2}$, then
			$\frac{10}{9} \le v \le \frac{2}{\sqrt{3}(t+1)y}$, as $\tfrac{y}{3} < |x|$ and $\operatorname{Im}\gamma z \ge \tfrac{\sqrt{3}(t+1)}{2}$. We deduce from the above that
			\begin{align}\label{eq-SSC24}
				\notag
				\int_{\substack{m(1,0)\\|x|>\frac{y}{3}}}\left|\frac{g_{t}(z)}{g_t(iy)}\right|dx&\le \frac{y}{2}1.02\exp\left(-\frac{\pi t}{12(t+1)y}\right) \int_{\frac{10}{9}}^{\frac{2}{\sqrt{3}(t+1)y}}\exp\left(\frac{\pi t}{12(t+1)y}\frac{1}{v}\right)\frac{dv}{\sqrt{v-1}}\\
				\notag
				&\le 2y\exp\left(-\frac{\pi t}{12(t+1)y}\right) \int_{\frac{10}{9}}^{\frac{2}{\sqrt{3}(t+1)y}}\exp\left(\frac{\pi t}{12(t+1)y}\frac{1}{v}\right)dv\\
				&\le  \frac{4}{\sqrt{3}(t+1)}\exp\left(-\frac{\pi t}{120(t+1)y}\right),\,\left(\text{as the integrand is maximum at $v=\frac{10}{9}$}\right). 
			\end{align}
			Similarly,
			\begin{align*}
				\notag
				\int_{m(c,d)}\left|\frac{g_{t}(z)}{g_t(iy)}\right|dx &\le   1.02\exp\left(-\frac{\pi t}{12(t+1)y}\right) \int_{m(c,d)}\exp\left(\frac{\pi t}{12(t+1)}\text{Im}\gamma z\right)dx\\
				\notag
				&=1.02\exp\left(-\frac{\pi t}{12(t+1)y}\right) \int_{m(c,d)}\exp\left(\frac{\pi t}{12(t+1)y}\frac{y^{2}}{(cx+d)^{2}+c^{2}y^{2}}\right)dx .
			\end{align*}
			Letting $u=\frac{x+\frac{d}{c}}{y}$ and $A=\sqrt{\frac{2}{\sqrt{3}(t+1)c^{2}y}-1}$ and then using the above equation, we obtain 
			\begin{align*}
				\notag
				\int_{m(c,d)}\left|\frac{g_{t}(z)}{g_t(iy)}\right|dx&= 1.02y\exp\left(-\frac{\pi t}{12(t+1)y}\right) \int_{|u|<A}\exp\left(\frac{\pi t}{12c^{2}(t+1)y}\frac{1}{u^{2}+1}\right)du\\
				\notag
				&\le1.02 y\exp\left(-\frac{\pi t}{12(t+1)y}\right)\exp\left(\frac{\pi t}{12c^{2}(t+1)y}\right) \int_{|u|<A}du\\
				&\le 2.04 \sqrt{\frac{2y}{\sqrt{3}c^{2}(t+1)}} \exp\left(-\frac{3\pi t}{48(t+1)y}\right).
			\end{align*}
			As $A$ need to be positive, $\frac{2}{\sqrt{3}(t+1)c^{2}y}>1$. This condition forces $c<\sqrt{\frac{2}{\sqrt{3}(t+1)y}}=M$. Therefore,
			\begin{align}\label{eq-SSC25}
				\notag
				\sum_{c=2}^{M}\sum_{\substack{d=0\\(c,d)=1}}^{c-1}\int_{m(c,d)}\left|\frac{g_{t}(z)}{g_t(iy)}\right|dx &\le  2.04 \sqrt{\frac{2y}{\sqrt{3}(t+1)}} \exp\left(-\frac{3\pi t}{48(t+1)y}\right) \sum_{c=2}^{M}\frac{1}{c} \sum_{\substack{d=0\\(c,d)=1}}^{c-1}1\\
				\notag
				&\le  2.04 \sqrt{\frac{2y}{\sqrt{3}(t+1)}} \exp\left(-\frac{3\pi t}{48(t+1)y}\right) \sum_{c=2}^{M}\frac{\phi(c)}{c}\\
				\notag
				&<  2.04 \sqrt{\frac{2y}{\sqrt{3}(t+1)}} \exp\left(-\frac{3\pi t}{48(t+1)y}\right)M\quad\left(\text{as $\frac{\phi(c)}{c}<1$}\right)\\
				&=   \frac{4.08}{\sqrt{3}(t+1)} \exp\left(-\frac{3\pi t}{48(t+1)y}\right).
			\end{align}
			\newline
			Substituting the expressions from \eqref{eq-SSC24} and \eqref{eq-SSC25} in \eqref{eq-SSC23}, we deduce
			\begin{align*}
				\int_{\frac{y}{3}\le|x|\le \frac{1}{2}}\left|\frac{g_{t}(z)}{g_t(iy)}\right| &\le  \frac{4}{\sqrt{3}(t+1)}\exp\left(-\frac{\pi t}{120(t+1)y}\right)+\frac{4.08}{\sqrt{3}(t+1)} \exp\left(-\frac{3\pi t}{48(t+1)y}\right)+\exp\left(-\frac{\pi}{150y}\right)\\
				&<2\exp\left(-\frac{\pi}{150y}\right).
			\end{align*}
		\end{proof}
		This completes the proofs of Case \ref{claim-3.11} and Case \ref{claim-3.12}, and hence the proof of Proposition \ref{proposition-error}.      
	\end{proof}
	In \cite{tyler2024asymptotics}, Tyler derives following two different expansions for $\mu_k(z)$ (see (3.14) and (3.16)), which are useful depending on whether $\Im z$ is large\footnote{Throughout the article, we refer to the imaginary part as \emph{large} when it lies in $[1,\infty)$, and as \emph{small} when it lies in $(0,1)$.} or small.
	\begin{proposition}\label{eq-(2.1)}
		For a large imaginary part of $z$, we have the following formula:
		\begin{align*}
			\notag
			\mu_{k}(z)&=\sum_{n=1}^{\infty}z^{k+1}(2\pi in)^{k-1}\sigma(n)\exp(2\pi inz)- \begin{cases} 
				\frac{z^{2}}{24} & \mbox{if } k=0,1 \\ 
				0 & \mbox{if } k\ge 2 
			\end{cases} \quad\text{and}\\  
			\mu_{k}^{\prime}(z)&=\sum_{n=1}^{\infty}\left((2\pi inz)^{k+1}+(k+1)(2\pi inz)^{k}\right) \frac{\sigma(n)}{2\pi in}\exp(2\pi inz)-\begin{cases} 
				\frac{z}{12} & \mbox{if } k=0,1 \\ 
				0 & \mbox{if } k\ge 2. 
			\end{cases}\\ 
			\notag   
		\end{align*}
	\end{proposition}
	\begin{proposition}\label{eq-(2.12)}
		When the imaginary part of $z$ is small, we have the following formula:
		\begin{align*}
			\mu_{k}(z)&=\sum_{n=1}^{\infty}P_{k}\left(\frac{2\pi in}{z}\right)\sigma(n)\exp\left(-\frac{2\pi in}{z}\right)+\frac{(-1)^{k}k!}{24}+\frac{z}{4\pi i}\begin{cases} 
				\log(-iz) & \text{if } k=0 \\ 
				(-1)^{k-1}(k-1)! & \mbox{if } k\ge 1 
			\end{cases}\\
			\notag
			\text{and}\quad \mu_{k}^{\prime}(z)&=\sum_{n=1}^{\infty}Q_{k}\left(\frac{2\pi in}{z}\right)\frac{\sigma(n)}{2\pi in}\exp\left(-\frac{2\pi in}{z}\right)+\frac{1}{4\pi i}\begin{cases} 
				1+\log (-iz) & \mbox{if } k=0 \\ 
				(-1)^{k-1}(k-1)! & \mbox{if } k\ge 1, 
			\end{cases}          
		\end{align*}  
		where $P_0(w)=w^{-1}$ and $P_{k}(w)=(w-k)P_{k-1}(w)-wP^{\prime}_{k-1}(w)$ and $Q_k(w)=w^{2}(P_k(w)-P^{\prime}_k(w))$. For $k=0,1,2,3,4$, explicit values of $P_k$ and $Q_k$ are given in \cite{tyler2024asymptotics} (see (3.18)). 
	\end{proposition}
	\begin{proof}[\textbf{\boldmath Proof of Proposition \ref{eq-(2.1)}}]
		Since we know that 
		\begin{align}\label{eq-logeta}
			\log \eta(z)=-\sum_{n=1}^{\infty}\frac{\sigma(n)}{n}\exp\left(2\pi inz\right)+\frac{\pi iz}{12}. 
		\end{align}
		Therefore, from (\ref{demu(z)}) and the above equation, we obtain
		\begin{align*}
			\mu_{k}(z)&=-\frac{z^{k+1}}{2\pi i} \left(\frac{d}{dz}\right)^{k} \log\eta(z)\\
			&=\frac{z^{k+1}}{2\pi i}\sum_{n=1}^{\infty}\frac{\sigma(n)}{n} \left(\frac{d}{dz}\right)^{k}\exp(2\pi inz)-\frac{z^{k+1}}{24} \left(\frac{d}{dz}\right)^{k} z\\
			&=\sum_{n=1}^{\infty}z^{k+1}(2\pi in)^{k-1}\sigma(n)\exp(2\pi inz)- \begin{cases} 
				\frac{z^{2}}{24} & \mbox{if } k=0,1 \\ 
				0 & \mbox{if } k\ge 2. 
			\end{cases}
		\end{align*}
		This completes the proof.
	\end{proof}
	
	\begin{proof}[\textbf{\boldmath Proof of Proposition \ref{eq-(2.12)}}]
		Applying the functional equation of $\eta$ from (\ref{eq-eta modular}), we obtain
		\begin{align*}
			\log\eta(z)&=\log\eta\left(-\frac{1}{z}\right)-\frac{1}{2}\log(-iz).
		\end{align*}
		Using (\ref{eq-logeta}) for $\log(-1/z)$, we deduce
		\begin{align*}
			\log\eta(z)&=-\sum_{n=1}^{\infty}\frac{\sigma(n)}{n}\exp\left(-\frac{2\pi in}{z}\right)-\frac{\pi i}{12z}-\frac{1}{2}\log (-iz).
		\end{align*}
		Combining (\ref{demu(z)}) with the above equation, we obtain
		\begin{align*}
			\mu_{k}(z)&=\frac{z^{k+1}}{2\pi i}\sum_{n=1}^{\infty}\frac{\sigma(n)}{n}\left(\frac{d}{dz}\right)^{k} \exp\left(-\frac{2\pi in}{z}\right)+\frac{z^{k+1}}{24}\left(\frac{d}{dz}\right)^{k} \frac{1}{z}+\frac{z^{k+1}}{4\pi i} \left(\frac{d}{dz}\right)^{k} \log(-iz) \\
			&=\sum_{n=1}^{\infty}P_{k}\left(\frac{2\pi in}{z}\right)\sigma(n)\exp\left(-\frac{2\pi in}{z}\right)+\frac{(-1)^{k}k!}{24}+\frac{z}{4\pi i}\begin{cases} 
				\log(-iz) & \text{if } k=0 \\ 
				(-1)^{k-1}(k-1)! & \mbox{if } k\ge 1. 
			\end{cases}
		\end{align*}
	\end{proof}
		
		Next, we prove the following three lemmas, which are crucial in the proof of Theorem \ref{theorem-main}.
		\begin{lemma}\label{lemma-mu2}
			Let $y$ satisfy $0<y\le \frac{1}{10}$ and $t>3$.\\
			(i) For $(t+1)y\ge1$, we have 
			\begin{equation*}
				\frac{(t+1)}{16}<(t+1)\mu_{2}(iy)-\mu_{2}((t+1)iy)<\frac{(t+1)}{12}.
			\end{equation*}   
			(ii) For $(t+1)y<1$, we have
			\begin{equation*}
				\frac{t-1}{12}<(t+1)\mu_{2}(iy)-\mu_{2}((t+1)iy)<\frac{t+1}{12}.   
			\end{equation*}
		\end{lemma}
		\begin{proof}  By $(ii)$ of Lemma~3.2 in \cite{tyler2024asymptotics}, the function $y\mapsto \mu_{2}(iy)$ is a strictly decreasing bijection from $(0,\infty)$ onto $\left(0,\tfrac{1}{12}\right)$. Therefore,
			\begin{equation}\label{eq-mu1}
				(t+1)\mu_{2}(iy)-\mu_{2}((t+1)iy)<\frac{t+1}{12}.  
			\end{equation}
			$(i)$ 
			Note,
			\begin{align*}
				\mu_{2}(iy)&=\sum_{n=1}^{\infty}\left(\frac{2\pi n}{y}-2\right)\sigma(n) \exp\left(-\frac{2\pi n}{y}\right)+\frac{1}{12}-\frac{y}{4\pi} \quad(\text{follows from Proposition \ref{eq-(2.12)})}\\
				&>\frac{1}{12}-\frac{y}{4\pi}\ge\frac{1}{12}-\frac{1}{40\pi}\quad\left(\text{as $y\le \frac{1}{10}$}\right).  
			\end{align*}
			As $(t+1)y \ge 1$, by Proposition \eqref{eq-(2.1)}, we obtain
			\begin{align}\label{eq-Numexplanation}
				\notag
				\mu_2((t+1)iy)&=\sum_{n=1}^{\infty}((t+1)y)^{3}2\pi n\sigma(n)\exp(-2\pi n(t+1)y)\\&< \sum_{n=1}^{\infty}2\pi n^{3}\exp(-2\pi n)\le 0.04, 
			\end{align}
			where we have used the estimate $\sigma(n) < n^2$. The numerical bound is established by observing that the function $a(x) = x^3 \exp(-2\pi x)$ is positive and strictly decreasing for $x \ge 1$. Thus, the series may be bounded by $a(1)+\int_{0}^{\infty}a(u)du<0.04,$ where the integral is evaluated using the Gamma function.
			\newline
			Combining the above bounds, we obtain
			\begin{equation*}
				(t+1)\mu_{2}(iy)-\mu_{2}((t+1)iy) > (t+1)\left(\frac{1}{12}-\frac{1}{40\pi}-\frac{0.04}{t+1}\right)\ge \frac{t+1}{16} \quad(\text{ as $t\ge 4$}). 
			\end{equation*}
			Finally, from (\ref{eq-mu1}), we conclude that
			\begin{equation*}
				\frac{(t+1)}{16}<(t+1)\mu_{2}(iy)-\mu_{2}((t+1)iy)<\frac{(t+1)}{12}.  
			\end{equation*}  
			$(ii)$ Using Proposition \ref{eq-(2.12)}, we obtain the following bounds for $(t+1)y < 1$:
			\begin{align*}
				&\mu_{2}((t+1)iy)=\sum_{n=1}^{\infty}\left(\frac{2\pi n}{(t+1)y}-2\right)\sigma(n) \exp\left(-\frac{2\pi n}{(t+1)y}\right)+\frac{1}{12}-\frac{(t+1)y}{4\pi}\quad\text{and}\\
				&\mu_{2}(iy)=\sum_{n=1}^{\infty}\left(\frac{2\pi n}{y}-2\right)\sigma(n) \exp\left(-\frac{2\pi n}{y}\right)+\frac{1}{12}-\frac{y}{4\pi}.
			\end{align*}
			Therefore, we have
			\begin{align}\label{eq-mu2}
				(t+1)\mu_{2}(iy)-\mu_{2}((t+1)y)>\frac{t}{12}-K_{1},
			\end{align}
			where 
			\begin{align*}
				&K_1=\sum_{n=1}^{\infty}\left(\frac{2\pi n}{(t+1)y}-2\right)\sigma(n) \exp\left(-\frac{2\pi n}{(t+1)y}\right).
			\end{align*}
			Since the function $K_{1}$ attains its maximum at $(t+1)y = 1$, we obtain
			\begin{align*}
				K_1\le\sum_{n=1}^{\infty}(2\pi n-2)n^{2}\exp(-2\pi n)<0.04\quad(\text{similar to (\ref{eq-Numexplanation})}).
			\end{align*}
			Hence, from (\ref{eq-mu2}) and (\ref{eq-mu1}), we have
			\begin{align*}
				\frac{t-1}{12}<(t+1)\mu_{2}(iy)-\mu_{2}((t+1)y)< \frac{t+1}{12}.
			\end{align*}
		\end{proof}
		\begin{lemma}
			Suppose that $0<y\le \frac{1}{10}$ and $t>3$. Then we have 
			\begin{equation*}\label{lemma-mu3}
				\left|\frac{(t+1)\mu_{3}(iy)-\mu_{3}((t+1)iy)}{(t+1)\mu_{2}(iy)-\mu_{2}((t+1)iy)}\right|<6.  
			\end{equation*}
		\end{lemma}
		\begin{proof}
			From part $(iii)$ of Lemma 3.3 in \cite{tyler2024asymptotics}, the function $y\mapsto \mu_{3}(iy)$ is a monotonically increasing bijection from $(0,\infty)$ onto $\left(-\tfrac{1}{4},0\right)$. Therefore,
			\begin{equation}\label{eq-mu31}
				|(t+1)\mu_{3}(iy)-\mu_{3}((t+1)iy)|\le|(t+1)\mu_{3}(iy)|+|\mu_{3}((t+1)iy)|<\frac{t+2}{4}.    
			\end{equation}  
			Suppose that $(t+1)y \ge 1$. Then, using the above equation together with
			$(i)$ of Lemma~\ref{lemma-mu2}, we obtain
			\begin{equation*}
				\left|\frac{(t+1)\mu_{3}(iy)-\mu_{3}((t+1)iy)}{(t+1)\mu_{2}(iy)-\mu_{2}((t+1)iy)}\right|<\frac{t+2}{4}\frac{16}{t+1}<6\quad (\text{as $t>3$}) .  
			\end{equation*}
			When $(t+1)y < 1$, we combine (\ref{eq-mu31}) with $(ii)$ of Lemma~\ref{lemma-mu2} to obtain
			\begin{equation*}
				\left|\frac{(t+1)\mu_{3}(iy)-\mu_{3}((t+1)iy)}{(t+1)\mu_{2}(iy)-\mu_{2}((t+1)iy)}\right|<\frac{t+2}{4}\frac{12}{t-1}\le6.   
			\end{equation*}
			This completes the proof.
		\end{proof}
		\begin{lemma}\label{lemma-mu4}
			Assume that $z=x+iy$, where $0<y\le \frac{1}{10}$ and $|x|<\frac{y}{3}$. Then, for every $t>3$,
			\begin{equation*}
				\left|\frac{(t+1)\mu_{4}(z)-\mu_{4}((t+1)z)}{(t+1)\mu_{2}(iy)-\mu_{2}((t+1)iy)}\right|<36.   
			\end{equation*}         
		\end{lemma}
		\begin{proof}
			We divide our analysis into two cases: $(t+1)y\ge 1$ or $(t+1)y<1.$
			\newline
			We begin by considering the situation when $(t+1)y\ge 1.$
			From the explicit form of $\mu_{4}$ given in Proposition \ref{eq-(2.1)}, we observe that for a large imaginary part, 
			\begin{equation}\label{eq-mu41}
				\mu_{4}((t+1)z) =\sum_{n=1}^{\infty}((t+1)z)^{5}(2\pi in)^{3}\sigma(n)\exp(2\pi in(t+1)z).
			\end{equation}
			For a small imaginary part, Proposition \ref{eq-(2.12)} gives
			\begin{equation}\label{eq-mu42}
				\mu_{4}(z)=\sum_{n=1}^{\infty}\left(\left(\frac{2\pi in}{z}\right)^{3}-12\left(\frac{2\pi in}{z}\right)^{2}+36\left(\frac{2\pi in}{z}\right)-24\right)\sigma(n)\exp\left(-\frac{2\pi in}{z}\right)+1-\frac{3z}{2\pi i}.   
			\end{equation}
			So, from (\ref{eq-mu41}), we deduce
			\begin{equation}\label{eq-mu44}
				\begin{aligned}
					| \mu_{4}((t+1)z)| 
					&\le(t+1)^{5}((x^{2}+y^{2}))^{\frac{5}{2}}\sum_{n=1}^{\infty}(2\pi n)^{3}\sigma(n)\exp(-2\pi n(t+1)y)\\
					&\le \left(\frac{10}{9}\right)^{\frac{5}{2}}((t+1)y)^{5}(2\pi)^{3}\sum_{n=1}^{\infty}n^{5}\exp(-2\pi n(t+1)y) \quad\left(\text{as $|x|<\frac{y}{3}$}\right)\\
					&\le \left(\frac{10}{9}\right)^{\frac{5}{2}}(2\pi)^{3}\sum_{n=1}^{\infty}n^{5}\exp(-2\pi n)<2 \quad (\text{similar to (\ref{eq-Numexplanation})}),
				\end{aligned}
			\end{equation}
			since the maximum is attained when $(t+1)y=1$.\newline
			Let $r \in \mathbb{Z}^{+}$. Then, for $|x| < \frac{y}{3}$, we have
			\begin{equation}\label{eq-mu43}
				\begin{aligned}
					& \left|\left(\frac{2\pi in}{rz}\right)^{k}\right| =\frac{(2\pi n)^{k}}{r^{k}(x^{2}+y^{2})^{k/2}}<\frac{(2\pi n)^{k}}{(ry)^{k}} \quad \text{and}\\
					& \left|\exp\left(-\frac{2\pi in}{rz}\right)\right|= \left|\exp\left(-\frac{2\pi in(x-iy)}{r(x^{2}+y^{2})}\right)\right|=\exp\left(-\frac{2\pi ny}{r(x^{2}+y^{2})}\right)\le \exp\left(-\frac{9\pi n}{5ry}\right).
				\end{aligned}
			\end{equation}
			By taking the modulus of both sides of \eqref{eq-mu42} and applying \eqref{eq-mu43}, we obtain
			\begin{align}\label{eq-mu45}
				\notag
				|\mu_{4}(z)|
				&\le\sum_{n=1}^{\infty}\left(\frac{(2\pi n)^{3}}{y^{3}}+12\frac{(2\pi n)^{2}}{y^{2}}+36\frac{(2\pi n)}{y}+24\right)\sigma(n)\exp\left(-\frac{9\pi n}{5y}\right)+1+\frac{3y}{2\pi}\sqrt{\frac{10}{9}}\\
				\notag
				&\le\sum_{n=1}^{\infty}\left((20\pi n)^{3}+12(20\pi n)^{2}+36(20\pi n)+24\right)n^{2}\exp\left(-18\pi n\right)+1+\frac{3}{2\pi}\sqrt{\frac{1}{90}}\quad\left(\text{as $0<y\le \frac{1}{10}$}\right)\\
				&<1.5 \quad (\text{similar to (\ref{eq-Numexplanation})}).
			\end{align}
			So, from (\ref{eq-mu44}) and (\ref{eq-mu45}), it follows that
			\begin{align*}
				|(t+1)\mu_{4}(z)-\mu_{4}((t+1)z)|&\le |(t+1)\mu_{4}(z)|+|\mu_{4}((t+1)z)|\\&<1.5(t+1)+2\le2(t+1)\quad(\text{as $t>3$}).   
			\end{align*}
			By $(i)$ of Lemma \ref{lemma-mu2}, we have
			\begin{equation*}
				\left|\frac{(t+1)\mu_{4}(z)-\mu_{4}((t+1)z)}{(t+1)\mu_{2}(iy)-\mu_{2}((t+1)iy)}\right|<2(t+1)\frac{16}{(t+1)}<36.    
			\end{equation*}
			We now consider the case where $(t+1)y<1.$ Using Proposition \ref{eq-(2.12)}, we deduce
			\begin{align}\label{eq-mu46}
				|(t+1)\mu_{4}(z)-\mu_{4}((t+1)z)|&= |E_1-E_2+t|\le |E_1|+|E_2|+t,
			\end{align}
			where
			\begin{align*}
				&E_{1}=(t+1)\sum_{n=1}^{\infty}\left(\left(\frac{2\pi in}{z}\right)^{3}-12\left(\frac{2\pi in}{z}\right)^{2}+36\left(\frac{2\pi in}{z}\right)-24\right)\sigma(n)\exp\left(-\frac{2\pi in}{z}\right)\,\,\text{and  }\\
				&E_2= \sum_{n=1}^{\infty}\left(\left(\frac{2\pi in}{(t+1)z}\right)^{3}-12\left(\frac{2\pi in}{(t+1)z}\right)^{2}+36\left(\frac{2\pi in}{(t+1)z}\right)-24\right)\sigma(n)\exp\left(-\frac{2\pi in}{(t+1)z}\right).\\
			\end{align*}
			It follows that the sum in \eqref{eq-mu42} is given by $E_1/(t+1)$, where $|E_{1}| < 0.005(t+1)$ is based on a calculation similar to that of \eqref{eq-mu45}. Applying \eqref{eq-mu43}, we obtain
			\begin{align*}
				|E_{2}| &\le\sum_{n=1}^{\infty}\left(\frac{(2\pi n)^{3}}{((t+1)y)^{3}}+12\frac{(2\pi n)^{2}}{((t+1)y)^{2}}+36\frac{(2\pi n)}{((t+1)y)}+24\right)\sigma(n)\exp\left(-\frac{9\pi n}{5(t+1)y}\right).  
			\end{align*}
			Since the maximum occurs at $(t+1)y = 1$ and $\sigma(n)<n^{2}$. Therefore,
			\begin{align*}
				|E_2|&\le\sum_{n=1}^{\infty}\left((2\pi n)^{3}+12(2\pi n)^{2}+36(2\pi n)+24\right)n^{2}\exp\left(-\frac{9\pi n}{5}\right)<5.5.
			\end{align*}
			Hence, from (\ref{eq-mu46}), we deduce
			\begin{align*}
				|(t+1)\mu_{4}(z)-\mu_{4}((t+1)z)|<1.5t+6\le 2t. 
			\end{align*}
			By $(ii)$ of Lemma \ref{lemma-mu2}, we simplify
			\begin{align*}
				\left|\frac{(t+1)\mu_{4}(z)-\mu_{4}((t+1)z)}{(t+1)\mu_{2}(iy)-\mu_{2}((t+1)iy)}\right|<2t\frac{12}{(t-1)}<36.    
			\end{align*}
		\end{proof}
		In the next theorem, we establish a general asymptotic formula for $p(N,t)$. Throughout the theorem, the symbol $\vartheta$ denotes a complex number with $|\vartheta| \le 1$. The value of $\vartheta$ may depend on the parameters and may vary at each occurrence. Since, in the proof of $(i)$ of Theorem \ref{theorem-1.1}, we established that $y$ is a solution of \eqref{eq-(1.3)}, we may choose such a $y$ in the following theorem.
		\begin{theorem}\label{theorem-main}
			Let $L=N+\frac{t}{24}$ and $y$ be chosen such that 
			\begin{equation}\label{eq-assumL}
				\left|\frac{\mu_{1}((t+1)iy)-(t+1)\mu_{1}(iy)}{(t+1)y^{2}}- L  \right|<\frac{2}{25y}   
			\end{equation}
			and $\frac{1}{y}\ge1000$. Then
			\begin{align*}
				p(N,t)&=\frac{\sqrt{t+1}y^{\frac{3}{2}}\exp\left(2\pi y\left(N+\frac{t}{24}\right)\right)g_{t}(iy)}{\sqrt{(t+1)\mu_{2}(iy)-\mu_{2}((t+1)iy)}}\left(1+ \frac{3.5(t+1)y}{(t+1)\mu_{2}(iy)-\mu_{2}((t+1)iy)}\vartheta\right),  
			\end{align*}
			where $|\vartheta|\le 1.$
		\end{theorem}
		\begin{proof}
			Recall (\ref{eq-pNt}), and apply Proposition \ref{proposition-error} to get
			\begin{align}\label{eq-p(N,t)}
				\notag
				p(N,t)&=\exp(2\pi Ly)g_{t}(iy)\int_{-y/3}^{y/3}\exp\left(-2\pi i Lx+2\pi i\frac{1}{2\pi i}\log\frac{g_{t}(z)}{g_{t}(iy)}\right)dx\\
				&+\exp(2\pi Ly)g_{t}(iy)\left(2\vartheta\exp\left(-\frac{\pi}{150y}\right)\right).
			\end{align}
			From the Taylor expansion of $\frac{1}{2\pi i}\log\frac{g_t(z)}{g_t(iy)}$ in (\ref{eq-taylor}), we have
			\begin{align}\label{eq-tayfinal}
				\notag
				\frac{1}{2\pi i}\log \frac{g_{t}(z)}{g_{t}(iy)}&=x\frac{(t+1)\mu_{1}(iy)-\mu_{1}((t+1)iy)}{(t+1)(iy)^{2}}\\
				&+\frac{x^{2}}{2!}\frac{(t+1)\mu_{2}(iy)-\mu_{2}((t+1)iy)}{(t+1)(iy)^{3}}\left(1+2i\xi\frac{x}{y}+3\vartheta \frac{x^{2}}{y^{2}}\right),
			\end{align}
			where $\xi\in (-1,1)$ due to Lemma \ref{lemma-mu3} and $|\vartheta|\le 1$ by Lemma \ref{lemma-mu4}.
			\newline
			Let 
			\begin{equation}\label{eq-alpha}
				\alpha=\frac{(t+1)\mu_{2}(iy)-\mu_{2}((t+1)iy)}{(t+1)y}.
			\end{equation}
			From Lemma \ref{lemma-mu2}, and using the fact that $\frac{1}{y}\ge 1000$, we have 
			\begin{align*}
				\alpha>\text{min}\left(\frac{1000}{16},\frac{1000(t-1)}{12(t+1)}\right)\ge 50>38,
			\end{align*}
			for $t>3$. 
			\newline
			Again, let 
			\begin{equation*}
				\beta=y\left(\frac{(t+1)\mu_{1}(iy)-\mu_{1}((t+1)iy)}{(t+1)(iy)^{2}}-L\right).   
			\end{equation*} 
			By assumption (\ref{eq-assumL}), we have $|\beta|<\frac{2}{25}.$ Substituting $\frac{1}{2\pi i}\log\frac{g_t(z)}{g_t(iy)}$ from \eqref{eq-tayfinal} in the integrand of \eqref{eq-p(N,t)}, we obtain
			\begin{align}\label{eq-exppart}
				\notag
				\exp\left(-2\pi i Lx+2\pi i\frac{1}{2\pi i}\log\frac{g_{t}(z)}{g_{t}(iy)}\right)&=\exp\left(\frac{2\pi ix}{y}y\left(\frac{(t+1)\mu_{1}(iy)-\mu_{1}((t+1)iy)}{(t+1)(iy)^{2}}-L\right)\right)\\
				\notag
				&\times\exp\left(-\pi\frac{x^{2}}{y^{2}}\frac{(t+1)\mu_{2}(iy)-\mu_{2}((t+1)iy)}{(t+1)y}\left(1+2i\xi\frac{x}{y}+3\vartheta \frac{x^{2}}{y^{2}}\right)\right)\\
				&=\exp\left(\frac{2\pi i\beta x}{y}\right)\exp\left(-\pi\alpha\left(\frac{x}{y}\right)^{2}\left(1+2i\xi\frac{x}{y}+3\vartheta \frac{x^{2}}{y^{2}}\right)\right).
			\end{align}
			After making the change of variable $u=\frac{x}{y}$, we deduce
			\begin{align*}
				\int_{-y/3}^{y/3}\exp\left(-2\pi i Lx+2\pi i\frac{1}{2\pi i}\log\frac{g_{t}(z)}{g_{t}(iy)}\right)dx&=y\int_{-1/3}^{1/3}\exp\left(2\pi i\beta u\right)\exp\left(-\pi\alpha u^{2}\left(1+2i\xi u+3\vartheta u^{2}\right)\right)du\\
				&=\frac{y}{\sqrt{\alpha}}\left(1+\frac{3.45}{\alpha}\vartheta\right),
			\end{align*}
			follow from Lemma 3.1 of \cite{tyler2024asymptotics}.
			Substituting in (\ref{eq-p(N,t)}), we obtain 
			\begin{align}\label{eq-alphaerr}
				\notag
				p(N,t)&=\exp(2\pi Ly)g_{t}(iy)\left(\frac{y}{\sqrt{\alpha}}\left(1+\frac{3.45}{\alpha}\vartheta\right)+2\vartheta\exp\left(-\frac{\pi}{150y}\right)\right)\\
				&=\exp(2\pi Ly)g_{t}(iy)\frac{y}{\sqrt{\alpha}}\left(\left(1+\frac{3.45}{\alpha}\vartheta\right)+2\vartheta\frac{\sqrt{\alpha}}{y} \exp\left(-\frac{\pi}{150y}\right)\right).
			\end{align}
			From Lemma \ref{lemma-mu2}, we obtain 
			\begin{align*}
				\alpha =\frac{(t+1)\mu_{2}(iy)-\mu_{2}((t+1)iy)}{(t+1)y}< \frac{(t+1)}{12(t+1)y}=\frac{1}{12y}.
			\end{align*}
			Therefore,
			\begin{align*}
				2\frac{{\alpha^{\frac{3}{2}}}}{y} \exp\left(-\frac{\pi}{150y}\right)&<\frac{2}{(12)^{\frac{3}{2}}y^{\frac{5}{2}}} \exp\left(-\frac{\pi}{150y}\right)\\
				&\le \frac{2(1000)^{\frac{5}{2}}}{(12)^{\frac{3}{2}}} \exp\left(-\frac{1000\pi}{150}\right)<0.05,
			\end{align*}
			as $ \frac{1}{y}\ge 1000$. So, $ 2\frac{\sqrt{\alpha}}{y} \exp\left(-\frac{\pi}{150y}\right)<\frac{0.05}{\alpha}.$ Therefore, from (\ref{eq-alphaerr}), we have
			\begin{align*}
				p(N,t) &=\exp(2\pi Ly)g_{t}(iy)\frac{y}{\sqrt{\alpha}}\left(1+\frac{3.5}{\alpha}\vartheta\right).   
			\end{align*}
			Substituting $\alpha$ from (\ref{eq-alpha}) in the above equation, we obtain our required result.
		\end{proof} 
		
		\section{Proofs of Main Results}\label{sec-proofre}
		In this section, we will prove all the results presented in the introduction. We proceed to establish bounds on $y$ for which an asymptotic expression for $p(N,t)$ holds, using the saddle point method.
		
		\begin{proof}[\textbf{\boldmath Proof of  \hyperlink{thm1.1i}{(i)} of Theorem \ref{theorem-1.1}}]
			According to (\ref{eq-(1.1)}), the function $p(N,t)$ admits the following integral representation:
			\begin{equation*}
				p(N,t) = \int_{-1/2}^{1/2} \exp\left(-2\pi i z \left(N + \frac{t}{24}\right)\right) g_t(z) dx.
			\end{equation*}
			To apply the saddle point method, we define
			\begin{equation*}
				f(z) = -2\pi i z\left(N+\frac{t}{24}\right) +  \log g_t(z). 
			\end{equation*}
			Recall from Section \ref{section1} that the saddle point $y$ is determined by solving the equation $f'(iy) = 0$. Now, it follows from (\ref{eq-gtz}) that
			\begin{equation}\label{eq-(3.2)}
				2\pi i\left(N + \frac{t}{24}\right) = \frac{d}{dz} \log \eta((t+1)z) - \frac{d}{dz} \log \eta(z).
			\end{equation}
			From equation (\ref{demu(z)}), we have
			\begin{align*}
				\frac{d}{dz} \log \eta(z) &= -\frac{2\pi i}{z^{2}} \mu_{1}(z)\quad\text{and}\\
				\frac{d}{dz}\log \eta((t+1)z)&=-\frac{2\pi i}{(t+1)z^{2}}\mu_{1}((t+1)z).
			\end{align*}
			Using the above equations in (\ref{eq-(3.2)}) and then setting $z = iy$, we deduce
			\begin{equation*}
				\frac{\mu_{1}((t+1)iy) - (t+1)\mu_{1}(iy)}{(t+1)y^{2}} = N + \frac{t}{24}.
			\end{equation*}
			We will now demonstrate that the solution $y>0$ is unique.
			\newline
			For large $\Im z$, it follows from Proposition \ref{eq-(2.1)} that
			\begin{equation}\label{eq-(3.3)}
				\mu_{1}(iy) = \frac{y^{2}}{24} - \sum_{n=1}^{\infty} y^{2}\sigma(n)\exp(-2\pi n y).
			\end{equation}
			On the other hand, for small $\Im z$, Proposition \ref{eq-(2.12)} gives
			\begin{equation}\label{eq-(3.4)}
				\mu_{1}(iy) = -\frac{1}{24} + \frac{y}{4\pi} + \sum_{n=1}^{\infty} \sigma(n) \exp\left(-\frac{2\pi n}{y}\right).
			\end{equation}
			From (\ref{eq-(3.3)}), we have
			\begin{equation*}
				\lim_{y \to \infty} \frac{\mu_{1}((t+1)iy) - (t+1)\mu_{1}(iy)}{(t+1)y^{2}} = \frac{t}{24},
			\end{equation*}
			and from (\ref{eq-(3.4)}), we have
			\begin{equation*}
				\lim_{y \to 0^{+}} \frac{\mu_{1}((t+1)iy) - (t+1)\mu_{1}(iy)}{(t+1)y^{2}} = \infty.
			\end{equation*}
			Hence, for any positive integer $N$, there exists some $y > 0$ such that
			\begin{equation*}
				\frac{\mu_{1}((t+1)iy) - (t+1)\mu_{1}(iy)}{(t+1)y^{2}} = N + \frac{t}{24}.
			\end{equation*}
			We know from the explicit expression of $\mu_k$ given in Proposition \ref{eq-(2.1)} and Proposition \ref{eq-(2.12)}, that $\mu_{2}(z) = -2\mu_{1}(z) + z\mu_{1}^{\prime}(z)$\footnote{The proof proceeds by considering the cases $y \ge 1$ and $y < 1$, and follows directly from Proposition \ref{eq-(2.1)} and Proposition \ref{eq-(2.12)}.}.  Applying this with the lower bound of $(t+1)\mu_2(iy)-\mu_2((t+1)iy)$ from Lemma \ref{lemma-mu2}, we deduce
			\begin{align*}
				\frac{d}{dy}\left(\frac{\mu_{1}((t+1)iy) - (t+1)\mu_{1}(iy)}{(t+1)y^{2}}\right) 
				&= \frac{1}{t+1}  \frac{\mu_{2}((t+1)iy) - (t+1)\mu_{2}(iy)}{y^{3}} < 0.
			\end{align*}
			Therefore, there exists a unique value $y > 0$ that satisfies equation (\ref{eq-(1.3)}).
			
			Next, we establish bounds for $y$. It follows from (\ref{eq-(3.3)}) and (\ref{eq-(3.4)}) that
			\begin{equation*}
				\mu_1(iy)< \frac{y^{2}}{24} \quad\text{and}\quad  \mu_{1}(iy)>-\frac{1}{24}.
			\end{equation*}
			Therefore,
			\begin{equation}\label{eq-yupper}
				\frac{\mu_{1}((t+1)iy) - (t+1)\mu_{1}(iy)}{(t+1)y^{2}} < \frac{t+1}{24} + \frac{1}{24y^{2}}.
			\end{equation}
			Given 
			\begin{align}\label{eq-ysolution}
				\frac{\mu_{1}((t+1)iy) - (t+1)\mu_{1}(iy)}{(t+1)y^{2}}=N+\frac{t}{24}.
			\end{align}
			So, it follows from (\ref{eq-yupper}) that
			\begin{align}\label{eq-(3.5)}
				y < \frac{1}{\sqrt{24N - 1}}.
			\end{align}
			Again, applying (\ref{eq-(3.4)}) and assuming that $(t+1)y<1$, we obtain
			\begin{equation*}
				\mu_{1}((t+1)iy) - (t+1)\mu_{1}(iy) = \frac{t}{24}+E_3-E_4 , 
			\end{equation*}
			where $E_3=\sum_{n=1}^{\infty}\sigma(n)\exp\left(-\frac{2\pi n}{(t+1)y}\right)$ and $E_4=\sum_{n=1}^{\infty}(t+1)\sigma(n)\exp\left(-\frac{2\pi n}{y}\right)$. One can easily prove that $E_3 \ge E_4$ for $(t+1)y < 1$. Therefore, 
			\begin{align*}
				\frac{\mu_{1}((t+1)iy) - (t+1)\mu_{1}(iy)}{(t+1)y^{2}} > \frac{t}{24(t+1)y^{2}}.
			\end{align*}
			Similarly, from (\ref{eq-ysolution}) and the above inequality, we obtain
			\begin{equation}\label{eq-(3.6)}
				y>\sqrt{\frac{t}{24(t+1)\left(N + \frac{t}{24}\right)}} .   
			\end{equation} 
			Combining (\ref{eq-(3.5)}) and (\ref{eq-(3.6)}), we deduce
			\begin{align*}
				\sqrt{\frac{t}{24(t+1)\left(N + \frac{t}{24}\right)}} <y < \frac{1}{\sqrt{24N - 1}}.  
			\end{align*}
		\end{proof}
		We next prove the second part of our theorem, which establishes an asymptotic result for $p(N,t).$
		\begin{proof}[\textbf{\boldmath Proof of  \hyperlink{thm1.1ii}{{(ii)}} of Theorem \ref{theorem-1.1}}] 
			Assume $N$ is sufficiently large. From Lemma \ref{lemma-mu2}, we have
			\begin{align}
				\frac{(t+1)y}{(t+1)\mu_{2}(iy)-\mu_{2}((t+1)iy)}\ll y.   
			\end{align} 
			From (\ref{eq-(1.3)}), $y$ is a solution of  $\frac{\mu_{1}((t+1)iy)-(t+1)\mu_{1}(iy)}{(t+1)y^{2}}- L=0$. Thus, the condition of Theorem \ref{theorem-main} is satisfied. Using (\ref{eq-gtz}) in Theorem \ref{theorem-main}, we obtain
			\begin{align*}
				p(N,t)&=\frac{\sqrt{t+1}y^{\frac{3}{2}} \exp\left(2\pi y\left(N+\frac{t}{24}\right)\right)\,\eta((t+1)iy)}{\sqrt{(t+1)\mu_{2}(iy)-\mu_{2}((t+1)iy)}\,\eta(iy)}\left(1+O(y)\right).   
			\end{align*}
		\end{proof}
		In the following, we derive explicit expressions for $p(N,t)$ that correspond to different ranges of $t$.
		\begin{proof}[\textbf{\boldmath Proof of \hyperlink{thm1.2i}{{(i)}} of Theorem \ref{theorem-1.2}}]
			From (\ref{eq-(1.4)}), we have $(t+1)y < 1$ for the given range of $t$. Using Theorem~\ref{theorem-1.1}, together with the bounds for $\eta((t+1)iy)$ and $\eta(iy)$ from Lemma~\ref{lemma-2.2}, we obtain
			\begin{align}\label{eq-pNT}
				\notag
				p(N,t)=&\frac{y^{\frac{3}{2}}\exp\left(2\pi y\left(N+\frac{t}{24}\right)-\frac{\pi}{12(t+1)y}-\rho_{1}\exp\left(-\frac{2\pi}{(t+1)y}\right)\right)}{\sqrt{(t+1)\mu_{2}(iy)-\mu_{2}((t+1)iy)}\,\exp\left(-\frac{\pi}{12y}-\rho_{2}\exp\left(-\frac{2\pi}{y}\right)\right)}\left(1+O(y)\right)\\
				=&\frac{y^{\frac{3}{2}}\exp\left(2\pi y\left(N+\frac{t}{24}\right)+\frac{\pi t}{12(t+1)y}+\rho_{2}\exp\left(-\frac{2\pi}{y}\right)\right)}{\sqrt{(t+1)\mu_{2}(iy)-\mu_{2}((t+1)iy)}\,\exp\left(\rho_1\exp\left(-\frac{2\pi}{(t+1)y}\right)\right)}\left(1+O(y)\right),
			\end{align}
			where $1 < \rho_1, \rho_2 < 1.00873$.
			\newline
			From $(i)$ of Theorem (\ref{theorem-1.1}), we know that $y$ is a solution of 
			\begin{align}\label{eq-soly}
				\frac{\mu_{1}((t+1)iy) - (t+1)\mu_{1}(iy)}{(t+1)y^{2}} = N + \frac{t}{24}. 
			\end{align}
			By Proposition \ref{eq-(2.12)}, we have the following bounds for $\mu_1(z)$ when $(t+1)y<1$:
			\begin{align*}
				\mu_{1}((t+1)iy)&=\sum_{n=1}^{\infty}\sigma(n)\exp\left(-\frac{2\pi n}{(t+1)y}\right) -\frac{1}{24}+\frac{(t+1)y}{4\pi}\quad\text{and} \\
				\mu_{1}(iy)&=\sum_{n=1}^{\infty}\sigma(n)\exp\left(-\frac{2\pi n}{y}\right) -\frac{1}{24}+\frac{y}{4\pi}.   
			\end{align*}
			Substituting the expressions for $\mu_1$ in (\ref{eq-soly}), we deduce the following
			\begin{align}\label{eq-AB}
				y^{2}(t+1)\left(N+\frac{t}{24}\right)=\frac{t}{24}+ B_1-B _2, 
			\end{align}
			where
			\begin{align*}
				B_1=\sum_{n=1}^{\infty}\sigma(n)\exp\left(-\frac{2\pi n}{(t+1)y}\right)\,\,\text{and  } B_2=(t+1) \sum_{n=1}^{\infty}\sigma(n)\exp\left(-\frac{2\pi n}{y}\right).  
			\end{align*}
			After substituting the value of $B_1$\footnote{Note that $B_1=\sum_{n=1}^{\infty}\sigma(n)\exp\left(-\frac{2\pi n}{(t+1)y}\right)=O\left(\left(N+\frac{t}{24}\right)^{-(\frac{1}{2}+\epsilon)}\right)$ for the given range   $2<(t+1)\le\frac{2\pi}{(\frac{1}{2}+\epsilon)\log \left(N+\frac{t}{24}\right)}\sqrt{\frac{24(t+1)\left(N+\frac{t}{24}\right)}{t}}$.} and $B_2$\footnote{We also note that $B_2=(t+1) \sum_{n=1}^{\infty}\sigma(n)\exp\left(-\frac{2\pi n}{y}\right)=O\left(\left(N+\frac{t}{24}\right)^{-2}\right)$.} in (\ref{eq-AB}) for the given range of t, and solving the quadratic equation in $y$, we have
			\begin{align}\label{eq-solutiony}
				\notag
				y&=\sqrt{\frac{t}{24(t+1)\left(N + \frac{t}{24}\right)}}\sqrt{1+\frac{24(B_1-B_2)}{t}}\\
				&= \sqrt{\frac{t}{24(t+1)\left(N + \frac{t}{24}\right)}}+O\left(N^{-\frac{3}{2}}\right).  
			\end{align}
			Similarly,
			\begin{align}\label{eq-yinvese}
				y^{-1}=\sqrt{\frac{24(t+1)\left(N + \frac{t}{24}\right)}{t}}+O\left(N^{-\frac{1}{2}}\right).    
			\end{align}
			For the above values of $y$ and $y^{-1}$, we obtain the following estimates:
			\begin{equation}\label{eq-estimates}
				\begin{aligned}
					&2\pi y\left(N+\frac{t}{24}\right)+\frac{\pi t}{12(t+1)y}=\frac{2\pi}{\sqrt{6}}\left(\sqrt{\frac{t}{t+1}\left(N+\frac{t}{24}\right)}\right)+O(N^{-\frac{1}{2}}),\\
					&\exp\left(\rho_2\exp\left(-\frac{2\pi}{y}\right)\right)=1+O\left(N^{-\frac{1}{2}}\right)\quad \text{and}\\
					&\exp\left(\rho_1\exp\left(-\frac{2\pi}{(t+1)y}\right)\right)=\exp\left(\rho_1\exp\left(-2\pi\sqrt{\frac{(24N+t)}{t(t+1)}}+O\left(N^{-\frac{1}{2}}\right)\right)\right)\left(1+O\left(N^{-\frac{1}{2}}\right)\right).
				\end{aligned}
			\end{equation}
			In the following claim, we obtain an explicit bound for $(t+1)\mu_{2}(iy)-\mu_{2}((t+1)iy)$, since our objective is to prove an asymptotic bound for $p(N,t)$ in the given range of $t$.
			\begin{claim}\label{claim-1}
				For all $\epsilon>0$, if $(t+1)<\frac{2\pi}{(\frac{1}{2}+\epsilon)\log \left(N+\frac{t}{24}\right)}\sqrt{\frac{24(t+1)\left(N+\frac{t}{24}\right)}{t}}$, then
				\begin{equation*}
					(t+1)\mu_{2}(iy)-\mu_{2}((t+1)iy)=\frac{t}{12}\left(1+O\left(N^{-\frac{1}{2}}\right)\right).
				\end{equation*}
			\end{claim}
			\begin{proof}[Proof of Claim \ref{claim-1}]
				From (\ref{eq-solutiony}), we see that $(t+1)y<1$ for the chosen range of $t$. Hence, we can apply Proposition \ref{eq-(2.12)}, which gives
				\begin{align*}
					\mu_{2}(iy)&=\sum_{n=1}^{\infty}\left(\frac{2\pi n}{y}-2\right)\sigma(n) \exp\left(-\frac{2\pi n}{y}\right)+\frac{1}{12}-\frac{y}{4\pi} \quad\text{and}  \\
					\mu_{2}((t+1)iy)& = \sum_{n=1}^{\infty} \left( \frac{2\pi n}{(t+1)y} - 2 \right) \sigma(n) \exp\left( -\frac{2\pi n}{(t+1)y} \right) + \frac{1}{12} - \frac{(t+1)y}{4\pi}.
				\end{align*}
				Therefore,
				\begin{equation*}
					(t+1)\mu_{2}(iy) - \mu_{2}((t+1)iy) = \frac{t}{12}+(t+1)E_5-E_6, 
				\end{equation*} 
				where \begin{align*}
					E_5=\sum_{n=1}^{\infty}\left(\frac{2\pi n}{y}-2\right)\sigma(n) \exp\left(-\frac{2\pi n}{y}\right) \quad\text{and}\quad E_6= \sum_{n=1}^{\infty} \left( \frac{2\pi n}{(t+1)y} - 2 \right) \sigma(n) \exp\left( -\frac{2\pi n}{(t+1)y} \right). 
				\end{align*}
				Using $y^{-1}$ from (\ref{eq-yinvese}), we derive bounds for $E_5$ and $E_6$, $E_5=O(N^{-2})$ and $E_6=O\left(N^{-\frac{1}{2}}\right)$, and obtain
				\begin{align*}
					(t+1)\mu_{2}(iy) - \mu_{2}((t+1)iy) = \frac{t}{12}\left(1+O\left(N^{-\frac{1}{2}}\right)\right).   
				\end{align*}
				This proves the claim.
			\end{proof}
			Inserting the value of $y$ and the estimates from (\ref{eq-estimates}) in the formula for $p(N,t)$ in (\ref{eq-pNT}), together with the preceding claim, we deduce
			\begin{align*}
				p(N,t)=&\frac{\sqrt{12}t^{\frac{1}{4}}((t+1)(24N+t))^{-\frac{3}{4}}\exp\left(\frac{2\pi}{\sqrt{6}}\sqrt{\frac{t}{t+1}\left(N+\frac{t}{24}\right)}\right)}{\exp\left(\rho_1\exp\left(-2\pi\sqrt{\frac{(24N+t)}{t(t+1)}}+O\left(N^{-\frac{1}{2}}\right)\right)\right)}\left(1+O\left(N^{-\frac{1}{2}}\right)\right) \\
				=&\frac{\sqrt{12}A_{t}(N)\exp\left(\frac{2\pi}{\sqrt{6}}\sqrt{\frac{t}{t+1}\left(N+\frac{t}{24}\right)}\right)}{\exp\left(\rho_1\exp\left(-2\pi\sqrt{\frac{(24N+t)}{t(t+1)}}+O\left(N^{-\frac{1}{2}}\right)\right)\right)}\left(1+O\left(N^{-\frac{1}{2}}\right)\right).   
			\end{align*} 
			This completes the proof.
		\end{proof}
		\begin{remark}
			If $t$ is fixed, then (\ref{eq-AB}) implies that
			\begin{align*}
				y= \sqrt{\frac{t}{24(t+1)\left(N + \frac{t}{24}\right)}}+O\left(N^{-2}\right).    
			\end{align*}
			Substituting this expression for $y$ in the proof of $(i)$ of Theorem \ref{theorem-1.2}, it follows that
			\begin{align*}
				p(N,t) = \sqrt{12}\,A_{t}(N)
				\exp\left(\frac{2\pi}{\sqrt{6}} \sqrt{\frac{t}{t+1} \left(N + \frac{t}{24}\right)}\right) 
				\left(1 + O\left(N^{-\frac{1}{2}}\right)\right)  
			\end{align*}
			which coincides with Theorem \ref{corollary-1.3}.
		\end{remark}
		\begin{proof}[\textbf{\boldmath Proof of \hyperlink{thm1.2ii}{{(ii)}} of Theorem \ref{theorem-1.2}}]  In the proof of part $(ii)$, we derive an upper bound for $p(N,t)$ instead of an exact bound. This relaxation is due to the difficulty of simplifying $p(N,t)$ with the explicit value of $y$. From (\ref{eq-AB}), we have
			\begin{align*}
				y^{2}(t+1)\left(N+\frac{t}{24}\right)<\frac{t}{24}+\sum_{n=1}^{\infty}\sigma(n)\exp(-2\pi n)\le \frac{t}{24}+\frac{1}{48}. 
			\end{align*}
			This implies 
			\begin{align*}
				y<\sqrt{\frac{2t+1}{24(2t+2)\left(N+\frac{t}{24}\right)}}.  
			\end{align*}
			Thus, applying the lower bound for $(t+1)\mu_2(iy)-\mu_2((t+1)iy)$ from Lemma \ref{lemma-mu2}, the lower bound for $y$ from (\ref{eq-(1.4)}), and the upper bound from the above equation in (\ref{eq-pNT}), we may write 
			\begin{align*}
				p(N,t)\le \frac{\sqrt{12}(2t+1)^{\frac{3}{4}}\exp\left(\frac{2\pi}{\sqrt{6}}\sqrt{\frac{2t+1}{2t+2}\left(N+\frac{t}{24}\right)}\right)}{\sqrt{t-1}\left((24N+t)(2t+2)\right)^{\frac{3}{4}}\exp\left(\exp\left(-2\pi\sqrt{\frac{24N+t}{t(t+1)}}\right)\right)}\left(1+O\left(N^{-\frac{1}{2}}\right)\right).  
			\end{align*}
		\end{proof}
		Finally, we obtain an asymptotic formula for $p(N,t)$ when
		$t+1 > \sqrt{24N}$.
		\begin{proof}[\textbf{\boldmath Proof of  \hyperlink{thm1.2iii}{{(iii)}} of Theorem \ref{theorem-1.2}}]
			From Theorem \ref{theorem-1.1}, we have
			\begin{align*}
				p(N,t)&=\frac{\sqrt{t+1}y^{\frac{3}{2}} \exp\left(2\pi y\left(N+\frac{t}{24}\right)\right)\,\eta((t+1)iy)}{\sqrt{(t+1)\mu_{2}(iy)-\mu_{2}((t+1)iy)}\,\eta(iy)}\left(1+O(y)\right).   
			\end{align*}
			Given $(t+1)y \ge 1$, we plugin bounds for $\eta((t+1)iy)$ and $\eta(iy)$ from Lemma \ref{lemma-2.1} and Lemma \ref{lemma-2.2} respectively, to obtain 
			\begin{align}\label{eq-proofi}
				\notag
				p(N,t) =\ 
				& \frac{\sqrt{t+1} \, y^{2} \exp\left(2\pi y N +\frac{\pi yt}{12}-\frac{\pi yt}{12}-\nu_{1} \exp(-2\pi(t+1)y)\right)}{
					\sqrt{(t+1)\mu_{2}(iy) - \mu_{2}((t+1)iy)} \, \exp\left(-\frac{\pi}{12y} -\nu_{2} \exp\left(-\frac{2\pi}{y}\right)\right)}\left(1+O(y)\right)\\
				\ 
				=& \frac{\sqrt{t+1} \, y^{2} \exp\left(2\pi y N +\frac{\pi}{12y}+\nu_{2} \exp\left(-\frac{2\pi}{y}\right)\right)}{
					\sqrt{(t+1)\mu_{2}(iy) - \mu_{2}((t+1)iy)} \, \exp\left(\nu_{1} \exp(-2\pi(t+1)y) \right)}\left(1+O(y)\right),
			\end{align}
			where $1 < \nu_1, \nu_2 < 1.00873$.
			\newline
			By $(i)$ of Theorem (\ref{theorem-1.1}), we know that $y$ is a solution of 
			\begin{align}\label{eq-soly1}
				\frac{\mu_{1}((t+1)iy) - (t+1)\mu_{1}(iy)}{(t+1)y^{2}} = N + \frac{t}{24}. 
			\end{align}
			To solve $y$ explicitly, consider the following from  Proposition \ref{eq-(2.1)} and Proposition~\ref{eq-(2.12)} respectively:
			\begin{align*}
				\mu_{1}(i(t+1)y)&=-\sum_{n=1}^{\infty}(t+1)^{2}y^{2}\sigma(n)\exp(-2\pi n(t+1)y)+\frac{(t+1)^{2}y^{2}}{24},\quad \text{when  $(t+1)y\ge 1$}\\  
				\text{and}\quad	\mu_{1}(iy)&=\sum_{n=1}^{\infty}\sigma(n)\exp\left(-\frac{2\pi n}{y}\right) -\frac{1}{24}+\frac{y}{4\pi},\quad\text{when $y<1$}.   
			\end{align*} 
			Substituting these bounds in (\ref{eq-soly1}), we have
			\begin{align}\label{eq-soly2}
				y^{2}\left(N-\frac{1}{24}\right)+\frac{y}{4\pi}-\frac{1}{24}+\frac{K_3+K_4}{t+1}=0,
			\end{align}
			where
			\begin{align*}
				K_{3}=\sum_{n=1}^{\infty}(t+1)^{2}y^{2}\sigma(n)\exp(-2\pi n(t+1)y)\,\,\text{and  } K_{4}=(t+1)\sum_{n=1}^{\infty}\sigma(n)\exp\left(-\frac{2\pi n}{y}\right).
			\end{align*}
			Interpreting (\ref{eq-soly2}) as a quadratic equation in $y$, we obtain
			\begin{align*}
				y=-\frac{1}{8\pi\left(N-\frac{1}{24}\right)} +\frac{1}{\sqrt{24\left(N-\frac{1}{24}\right)}}\sqrt{1+\frac{3}{8\pi^{2}\left(N-\frac{1}{24}\right)}-\frac{24(K_3+K_4)}{t+1}}.   
			\end{align*}
			Substituting $y$ in $K_3$ \footnote{Note that  $K_5(N,t)=\frac{K_3}{(t+1)}=\frac{1}{(t+1)}\sum_{n=1}^{\infty}(t+1)^{2}y^{2}\sigma(n)\exp(-2\pi n(t+1)y)\le C_2N^{-1/2}$ for $t+1 > \sqrt{24N}$ and for some constant $C_{2}$ independent of $N$.  
				This estimate is used in the derivation of (\ref{eq-soly3}).} and $K_4$ \footnote{We also note that $K_6=\frac{K_4}{(t+1)}=\sum_{n=1}^{\infty}\sigma(n)\exp\left(-\frac{2\pi n}{y}\right)=O(N^{-2}).$} in the given range of $t$, we simplify
			\begin{equation}\label{eq-soly3}
				y=\frac{1}{\sqrt{24N}}-\frac{K_5(N,t)}{\sqrt{N}}+ \frac{C_{1}}{N}+O\left(\frac{1}{N^{\frac{3}{2}}}\right),  
			\end{equation}
			where $0<K_5(N,t)\le C_2N^{-1/2}$, for some constants $C_1$ and $C_2$ independent of $N$.
			\newline
			Similarly, we obtain an expression for $y^{-1}$ 
			\begin{align*}
				y^{-1}=\sqrt{24N}\left(1+\sqrt{24}K_5(N,t)-\frac{C_1\sqrt{24}}{\sqrt{N}}+O\left(\frac{1}{N}\right)\right).   
			\end{align*}
			The above formulas for $y$ and $y^{-1}$ give
			\begin{equation}\label{eq-yy-inverse}
				\begin{aligned}
					&2\pi Ny+\frac{\pi}{12y}=\frac{2\pi}{\sqrt{6}}\sqrt{N}+O\left(N^{-\frac{1}{2}}\right)\quad \text{and}\\
					&\exp\left(\nu_2\exp\left(-\frac{2\pi}{y}\right)\right)=1+O\left(N^{-\frac{1}{2}}\right). 
				\end{aligned}
			\end{equation}
			\begin{claim}\label{claim-2}
				Let $t + 1 > \sqrt{24N}$. Then
				\begin{equation*}
					(t+1)\mu_{2}(iy)-\mu_{2}((t+1)iy)=\frac{t+1}{12}\left(1+O\left(N^{-\frac{1}{2}}\right)\right).
				\end{equation*}
			\end{claim}
			\medskip
			\begin{proof}[Proof of Claim \ref{claim-2}]
				For given $y$, $\sum_{n=1}^{\infty}\left(\frac{2\pi n}{y}-2\right)\sigma(n) \exp\left(-\frac{2\pi n}{y}\right)=O\left(N^{-\frac{1}{2}}\right)$. Note by 
				Proposition \ref{eq-(2.12)},
				\begin{align*}
					(t+1)\mu_{2}(iy)&=(t+1)\sum_{n=1}^{\infty}\left(\frac{2\pi n}{y}-2\right)\sigma(n) \exp\left(-\frac{2\pi n}{y}\right)+\frac{t+1}{12}-\frac{(t+1)y}{4\pi}\\
					&=\frac{t+1}{12}\left(1+O\left(N^{-\frac{1}{2}}\right)\right).
				\end{align*}
				Since $(t+1)y \ge 1$, by Proposition \ref{eq-(2.1)} we obtain
				\begin{equation*}
					\mu_{2}((t+1)iy) = \sum_{n=1}^{\infty}((t+1)y)^{3}(2\pi n)\sigma(n)\exp(-2\pi n(t+1)y).
				\end{equation*}
				The maximum of $\mu_2((t+1)iy)$ is attained at $(t+1)y = 1$, and one can easily verify that $\mu_{2}((t+1)iy)=O(1).$ Hence,
				\begin{align*}
					(t+1)\mu_{2}(iy)-\mu_{2}((t+1)iy)=&\frac{t+1}{12}\left(1+O\left(\text{max}\left(N^{-\frac{1}{2}}, \frac{1}{t+1}\right)\right)\right)\\
					=&\frac{t+1}{12}\left(1+O\left(N^{-\frac{1}{2}}\right)\right).
				\end{align*}
				This concludes the proof.
			\end{proof}
			Substituting $y = \frac{1}{\sqrt{24N}}\left(1+O\left(N^{-\frac{1}{2}}\right)\right)$ from \eqref{eq-soly3}, together with the estimates from \eqref{eq-yy-inverse} and the bound for $(t+1)\mu_{2}(iy)-\mu_{2}((t+1)iy)$ from the preceding claim in \eqref{eq-proofi}, we simplify the expression for $p(N,t)$ as follows:
			\begin{align*}
				p(N,t)&=\frac{\exp\left(\frac{2\pi}{\sqrt{6}}\sqrt{N}\right)}{4\sqrt{3}N}\exp\left(-\nu_1\exp\left(-\frac{\pi (t+1)}{\sqrt{6N}}\left(1+O\left(N^{-\frac{1}{2}}\right)\right)\right)\right)\left(1+O\left(N^{-\frac{1}{2}}\right)\right)\\
				&=p(N)\exp\left(-\nu_1\exp\left(-\frac{\pi (t+1)}{\sqrt{6N}}\left(1+O\left(N^{-\frac{1}{2}}\right)\right)\right)\right)\left(1+O\left(N^{-\frac{1}{2}}\right)\right),
			\end{align*}
			where the second line follows from \eqref{eq-hardy}.
			\newline
			This completes the proof.
		\end{proof}
		
		\begin{proof}[\textbf{\boldmath Proof of Theorem \ref{theorem-1.4}}]
			By Proposition 2.7 of \cite{barman2025lower} and Theorem \ref{theorem-1.2}, the result follows immediately.
		\end{proof}
		
		\bibliographystyle{abbrv}					
		\bibliography{report1}  				
	\end{document}